\documentclass[10pt]{article}
\usepackage{amsmath}
\usepackage{mathrsfs}
\usepackage{amsfonts}
\usepackage{amssymb}
\usepackage{graphicx, times}
\usepackage{subfig}
\usepackage{amsthm}
\usepackage{bm}
\usepackage[section]{placeins}  
\newtheorem{The}{Theorem}[section]
\newtheorem{Def}{Definition}[section]

\newtheorem{Ex}{Example}[section]

\newtheorem{remark}{Remark}[section]
\newtheorem{coro}{Corollary}
\newtheorem{Conj}{Conjecture}

\newcommand{\cotinv}{\cot^{-1}}
\begin{document}

\begin{center}
{\LARGE {\bf Nonexistence of invariant manifolds in fractional }}\\[0.2cm]
{\LARGE {\bf order dynamical systems }}
\vskip 1cm
{\Large  Sachin Bhalekar, Madhuri Patil}\\
\textit{Department of Mathematics, Shivaji University, Kolhapur - 416004, India, Email:sachin.math@yahoo.co.in, sbb\_maths@unishivaji.ac.in (Sachin Bhalekar), madhuripatil4246@gmail.com (Madhuri Patil)}\\
\end{center}
\begin{abstract}
 Invariant manifolds are important sets arising in the stability theory of dynamical systems. In this article, we take a brief review of invariant sets. We provide some results regarding the existence of invariant lines and parabolas in planar polynomial systems. We provide the conditions for the invariance of linear subspaces in fractional order systems. Further, we provide an important result showing the nonexistence of invariant manifolds (other than linear subspaces) in fractional order systems.
\end{abstract}
\vskip 1cm
\noindent
\section{Introduction}
Dynamical systems \cite{Meiss,Hale,Hirsch,Perko} is a trending branch of Mathematics playing a vital role in the Mathematical Analysis as well as in the Applied Sciences \cite{Banerjee,Walker,Jackson,Strogatz}. Chaos theory and Fractals \cite{Alligood,Wiggins,Pietronero,Peitgen,Barnsley,Devaney} are the sub-branches of this theory which have attracted the attention of scientists as well as layman. The applications of dynamical systems are found in Arts \cite{Pachepsky,Crilly} and Social Sciences \cite{Kuang,Gori} also. The theoretical results such as Hartman-Grobman theorem \cite{Perko}, Stable manifold theorem \cite{Perko} and Poincare-Bendixson theorem \cite{Perko} made the substantial contributions to the Mathematical Analysis. 
\par Fractional calculus deals with the differentiation and integration of arbitrary order \cite{Love,Oldham,Samko,Samko-1995,Podlubny,Bagley,Diethelm,Das}. The fractional derivative operators are non-local and hence very useful in modeling the memory in the natural systems \cite{Hilfer,Magin,Sabatier,Baleanu,Mainardi,Petras,Monje}. Existence and uniqueness of solution of fractional order initial value problems are discussed in \cite{Delbosco, Gejji1, Wei, Diethelm}. Stability results of fractional order dynamical systems are proposed in \cite{Matignon,Deng,Tavazoei,Bhalekar}. Various efficient numerical methods such as fractional Adams method (FAM) \cite{DFF}, new predictor-corrector method (NPCM) \cite{V. Gejji 3}, new finite-difference predictor-corrector method \cite{Jhinga} and predictor-corrector methods \cite{Kumar} are proposed to handle the tedious computations involved in the nonlinear fractional differential equations (FDE). Attempts are made to construct invariant manifolds \cite{Cong,Sayevand,Deshpande,Deshpande1,Cong1,Ma,Wang} for fractional order systems. 
\par It should be noted that, one cannot expect the same behavior from fractional order systems (FOS) as their classical (integer-order) counterparts.\\
e.g. The solution trajectories of classical differential dynamical systems are smooth whereas those of FOS can have self-intersections \cite{Bhalekar1,Deshpande2}. Some other differences are given in \cite{Kaslik Shiv,Bhalekar2,S. Bhalekar 3}. It is very natural to expect the nonexistence of invariant manifolds from FOSs. 
\par The paper is organized as below:\\
Basic definitions and results are listed in Section \ref{Sec2}. In Section \ref{Sec3}, we propose the conditions for the existence of invariant straight lines, parabolas, some other polynomial curves and exponential curves under the flow of planar quadratic system of ODE. Section \ref{Sec4} provides the answer to the question: 
 Is there exists invariant manifolds for fractional order quadratic systems? In Section \ref{Sec5}, some comments are made on the nonexistence of invariant manifolds in FOSs. Conclusions are summarized in Section \ref{Sec6}.

\section{Preliminaries}\label{Sec2}
This section contains some basic definitions and results given in the literature.  
\begin{Def}\cite{Perko}
Let $E$ be an open subset of $\mathbb{R}^n$ and let $f\in C^1(E)$. For $X_0\in E$, let $\Phi (t,\,X_0)$ be the solution of the initial value problem
\begin{equation}
\dot{X}=f(X),  \label{4.2.1}
\end{equation}
$$X(0)=X_0 $$
defined on its maximal interval of existence $I(X_0)$. Then for $t\in I(X_0)$, the set of mappings $\Phi_t: E\rightarrow\mathbb{R}^n$ defined by
\begin{equation}
\Phi_t(X_0)=\Phi (t,\,X_0) \label{4.2.2}
\end{equation}
is called the flow of the differential equation (\ref{4.2.1}).
\end{Def}
\par Note that $\Phi_0=I$, the identity map.\\
For any $t$, $s$ $\in I(X_0)$,
\begin{equation}
 \Phi_t\circ\Phi_s=\Phi_{t+s}. \label{4.2.3}
\end{equation}
This is called semi-group property of the flow.

\begin{Def}
The set $S\subseteq\mathbb{R}^n$ is said to be invariant under the flow $\Phi_t:\mathbb{R}^n\rightarrow\mathbb{R}^n$ of system (\ref{4.2.1})
if\,\, $\Phi_t(S)\subseteq S$, $\forall$ $t\in I(X_0)$.
\end{Def}

\begin{Def}
A steady state solution of 
(\ref{4.2.1})
is called an equilibrium point. Thus, $X_*$ is an equilibrium point of (\ref{4.2.1}) if $f(X_*)={\bm 0}$.
\end{Def}
\par For the classification of equilibrium points, the readers are referred to  \cite{Perko,Meiss}.

\begin{Def}\cite{Perko} \label{Def 2.4}
Let $E$ be an open subset of\, $\mathbb{R}^n$ and let $f\in C^1(E)$.
The global stable manifold of the system (\ref{4.2.1}) corresponding to an equilibrium $X_*$ is defined as
\begin{equation}
S=\{{\bm c}\in\mathbb{R}^n : \lim\limits_{t\rightarrow\infty}\Phi_t({\bm c})=X_*, \,\, {\bm c}\in E\}.
\end{equation}
Note that, $S$ is invariant under $\Phi_t$.
\end{Def}

\begin{Def}\cite{Perko}
The Homoclinic orbit is an invariant set which is a  closed loop passing through a saddle equilibrium. Such loop is contained in the intersection of stable and unstable manifolds of a single equilibrium point. i.e. Homoclinic orbit approaches to the single equilibrium point as $t\rightarrow\pm\infty$.
\end{Def}

\begin{Def}\cite{Perko}
The Heteroclinic orbit connects different equilibrium points. It approaches different equilibrium points as $t\rightarrow\pm\infty$.
\end{Def}

{\bf Note:}
\begin{enumerate}
\item Every solution curve of (\ref{4.2.1}) is an invariant set.

\item In particular, if $f(X)=AX$, where $A$ is a square matrix then the eigenvectors of $A$ (straight lines) are invariant sets.

\item If $u\pm iv$ are complex eigenvalues of $A$ and if $W$ is a (complex) eigenvector corresponding to $u\pm iv$, then the linear subspace spanned by $\text{Re}(W)$ and $\text{Im}(W)$ is invariant under $\Phi_t$.

\item If $f$ is non-linear then we can have some other invariant sets. e.g. curve, surface  (manifolds).
\end{enumerate}

\begin{Def} \cite{Meiss}
Separatrix $S$ is an invariant manifold such that the qualitative properties of solutions change at $S$.\\
The (global) stable and unstable manifolds of saddle equilibrium are examples of separatrices. 
\end{Def}

\begin{Def}\cite{Podlubny} \label{Def 3,2.1}
Let $\alpha\ge0$ \,\, ($\alpha\in\mathbb{R}$). Then Riemann-Liouville (\text RL) fractional integral of a function $f\in C[0,b]$, $b>0$ of order `$\alpha$' is defined as,
\begin{equation}
{}_0\mathrm{I}_t^\alpha f(t)=
\frac{1}{\Gamma{(\alpha)}}\int_0^t (t-\tau)^{\alpha-1}f(\tau)\,\mathrm{d}\tau. \label{3.1}
\end{equation}
\end{Def}

\begin{Def}\cite{Podlubny}\label{Def 3,2.2}
The Caputo fractional derivative of order $\alpha>0$, $n-1<\alpha< n$, $n\in \mathbb{N}$ is defined for $f\in C^n[0,b]$,\, $b>0$ as,
\begin{equation}
{}_0^{C}\mathrm{D}_t^\alpha f(t)=
\begin{cases}
\frac{1}{\Gamma{(n-\alpha)}}\int_0^t (t-\tau)^{n-\alpha-1}f^{(n)}(\tau)\,\mathrm{d}\tau & \mathrm{if}\,\, n-1<\alpha< n\\
\frac{d^n}{dt^n}f(t) & \mathrm{if}\,\, \alpha=n.
\end{cases}\label{3.2}
\end{equation}
Note that ${}_0^{C}\mathrm{D}_t^\alpha c=0$, where $c$ is a constant.\\
 Thus, equilibrium points of    the fractional order systems ${}_0^C\mathrm{D}_t^\alpha X=f(X)$ are same as their classical counterparts (\ref{4.2.1}).
\end{Def}

\begin{The}
\cite{Luchko} The Solution of non-homogeneous fractional order differential equation
\begin{equation}
{}_0^C\mathrm{D}_t^\alpha x(t)+\lambda x(t)=g(t), \qquad 0<\alpha<1, \label{1.12}
\end{equation} 
is given by,
\begin{equation}
x(t)=\int_0^t \tau^{\alpha-1}E_{\alpha,\alpha}(-\lambda \tau^\alpha)g(t-\tau)\,\mathrm{d}\tau+x(0)E_\alpha(-\lambda t^\alpha), \label{1.13}
\end{equation}
where $E_\alpha(z)=\sum_{k=0}^\infty \frac{z^k}{\Gamma(\alpha k+1)}$\, \quad  and \\ \quad $E_{\alpha,\beta}(z)=\sum_{k=0}^\infty \frac{z^k}{\Gamma(\alpha k+\beta)}\, ,\quad z\in\mathbb{C}, \,\,(\alpha>0,\,\beta>0)$ \,\,\\ are Mittag-Leffler functions \cite{Podlubny}.
\end{The}

\begin{remark}
If a manifold is given by the equation
\begin{equation}
y=h(x), \quad x\in\mathbb{R}^n, \quad y\in\mathbb{R}^m \label{4.2.9}
\end{equation}
and the system of differential equations is given by,
\begin{equation}
\begin{split}
\dot{x}&=f(x,\,y)\\
\dot{y}&=g(x,\,y)\\ \label{4.2.16}
\end{split}
\end{equation}
then the condition 
\begin{equation}
\mathrm{D}h(x) \dot{x}=\dot{y} \quad\Rightarrow \mathrm{D}h(x) f(x,\,h(x))=g(x,\,h(x)),
\end{equation}

is necessary and sufficient to show the invariance of (\ref{4.2.9}) under the flow of (\ref{4.2.16}). This condition is known as tangency condition \cite{Wiggins}.
\end{remark}

\section{Some invariant manifolds of planar quadratic systems} \label{Sec3}
In this section, we provide some necessary and sufficient conditions to exist the invariant lines and invariant parabolas for planar polynomial systems with classical derivatives.

\subsection{Literature review}
Consider a planar polynomial vector field
\begin{equation}
\begin{split}
\dot{x}&=P_n(x,y),\\
\dot{y}&=Q_n(x,y),
\end{split} \label{4.12}
\end{equation}
where $P_n(x,y)$ and $Q_n(x,y)$ are polynomials of degree $n$.
\par The second part of Hilbert's sixteenth problem \cite{Hilbert} is related to the number of limit cycles in polynomial system (\ref{4.12}). The literature review of planar quadratic system is taken by Coppel \cite{Coppel}. In \cite{Chicone}, authors studied  the classification of phase portraits of a quadratic system in a region surrounded by separatrix  cycle.
\par In \cite{Yanqian}, Ye proposed the following conjecture:
\begin{Conj}\label{conj1}
When $n$ is odd, the system (\ref{4.12})
has at most $M_n=2n+2$ invariant lines; when $n$ is even, the system (\ref{4.12}) has at most $M_n=2n+1$ invariant straight lines.
\end{Conj}
For $n=2, 3$ and $4$, this conjecture is proved by Sokulski \cite{Sokulski}. However, the conjecture is false \cite{Artes} if $n>4$. It should be noted that the system (\ref{4.12}) can have infinitely many invariant straight lines (see Example \ref{Ex 3.5}).
\par Artes \cite{Artes} proposed the following important result:
\begin{The}
Assume that the polynomial differential system (\ref{4.12}) of degree $n$ has finitely many invariant straight lines. Then the following statements hold for system (\ref{4.12}).
\begin{enumerate}
\item Either all the points on an invariant line are equilibrium or the line contains no more than $n$ equilibrium points.
\item No more than $n$ invariant straight lines can be parallel.
\item The set of all invariant straight lines through a single point cannot have more than $n+1$ different slopes.
\item Either it has infinitely many finite equilibrium points, or it has at most $n^2$ finite equilibrium points.
\end{enumerate}
\end{The}

\subsection{Necessary and sufficient conditions for the existence of invariant straight lines}
In this section, we propose some necessary and sufficient conditions for the existence of invariant straight lines for the system (\ref{4.12}).
\begin{The}\label{Thm 4.3.1}
Consider planar polynomial system of degree $n$,
\begin{equation}
\begin{split}
\dot{x}&=\sum_{i,j=0}^{n}a_{i,j}x^iy^j\\
\dot{y}&=\sum_{i,j=0}^{n}b_{i,j}x^iy^j \label{4.3.2} 
\end{split}
\end{equation}
with $a_{0,0}=b_{0,0}=0$
\begin{enumerate}
\item There exists infinitely many invariant straight lines  $y=mx$ to the system (\ref{4.3.2})
if 
\begin{equation}
\begin{split}
b_{k,0}=a_{0,k}&=0 \quad \text{and}\\
b_{k-j,j}-a_{k-(j-1),j-1}&=0, \,\, 1\le j\le k \label{4.3.3}
\end{split}
\end{equation}
for all $k=1,2,\dots,n$.
\item Consider any set of distinct values $i_1,i_2,\dots,i_l$ from \{1,2,\dots,n\}, where $1\le l \le n$ and 
\begin{equation}
b_{k,0}=a_{0,k}=
b_{k-j,j}-a_{k-(j-1),j-1}=0,  \label{4.3.5}
\end{equation}
where, $k\in\{1,2,\dots,n\}-\{i_1,i_2,\dots,i_l\}$ and $1\le j \le k$. The values of $m$ obtained from the system of $l$ equations,
\begin{equation}
b_{i_p,0}+\sum_{j=1}^{i_p}\left(b_{i_p-j,j}-a_{i_p-(j-1),j-1}\right)m^j-a_{0,i_p}m^{i_p+1}=0,\quad 1\le p\le l
\end{equation}
will give the invariant lines $y=mx$.
\end{enumerate}

\end{The}
\begin{proof}
Consider the equation of line
\begin{equation}
y=mx. \label{4.3.6}
\end{equation}
Differentiating (\ref{4.3.6}), we get
$
 \dot{y}=m\dot{x}
$.\\
Therefore, the tangency condition implies that,
\begin{equation*}
\sum_{i,j=0}^{n}(b_{i,j}-ma_{i,j})m^jx^{i+j}=0\quad \forall x\in \mathbb{R}.
\end{equation*} 
This holds if and only if,
\begin{equation*}
\sum_{j=0}^{k}(b_{k-j,j}-ma_{k-j,j})m^j=0
\end{equation*}
for each $k=1,2,\dots,n$.
\begin{equation}
\Leftrightarrow b_{k,0}+\sum_{j=1}^{k}(b_{k-j,j}-a_{k-(j-1),j-1})m^j-a_{0,k}m^{k+1}=0 \label{4.3.7}
\end{equation}
for each $k=1,2,\dots,n$.\\
{\bf Case 1:} If $b_{k,0}=a_{0,k}=0$ and $b_{k-j,j}-a_{k-(j-1),j-1}=0$, $1\le j\le k$ for each $k=1,2,\dots,n$, then the tangency condition (\ref{4.3.7}) is satisfied by any $m\in \mathbb{R}$.\\
This proves the Statement 1.\\
{\bf Case 2:} Now, instead of equating all the coefficients of all the powers of $m$ in (\ref{4.3.7}) to zero, we solve some of the equations (\ref{4.3.7}) for $m$ and proceed as in Case 1  for other equations. For $1\le l \le n$, if we solve any $l$ equations (\ref{4.3.7}) for $m$ and equate coefficients of powers of $m$ to zero in the remaining equations, then we obtain the Statement 2.\\
Note that, the Statement 2 provides 
$$\binom{n}{1}+\binom{n}{2}+\cdots+\binom{n}{n}=2^n-1 $$
ways to find invariant straight lines for the system (\ref{4.3.2}).
\end{proof}

\begin{coro}
Consider the planar polynomial system (\ref{4.3.2}) of degree $n$. Then the lines $x=k$ (respectively, $y=l$) are invariant under the flow of system (\ref{4.3.2}) if and only if 
\begin{equation}
\sum_{i,j=0}^{n}a_{i,j}k^iy^j=0\quad \forall y\in\mathbb{R} \,\,\,(\text{respectively}, \sum_{i,j=0}^{n}b_{i,j}x^il^j=0\quad \forall x\in\mathbb{R}). \label{4.3.8}
\end{equation}
where $k$ and $l$ are real constants.
\end{coro}

 The Theorem \ref{Thm 4.3.1} is illustrated for $n=2$ in the following Theorem.
\begin{The} \label{Thm 4.3.2}
Consider the planar quadratic system,
\begin{equation}
\begin{split}
\dot{x}&=a_1x+a_2y+a_3x^2+a_4y^2+a_5xy\\
\dot{y}&=b_1x+b_2y+b_3x^2+b_4y^2+b_5xy \label{4.3.9} 
\end{split}
\end{equation}

\begin{enumerate}
\item The line 
(\ref{4.3.6})
is invariant under the flow of given system (\ref{4.3.9}), if and only if
\begin{enumerate}

\item $a_2=0$, \,$b_1=0$\, and \, $a_1=b_2$. In this case, the real values of $m$ obtained from the cubic equation,
\begin{equation*}
a_4m^3+(a_5-b_4)m^2+(a_3-b_5)m-b_3=0
\end{equation*}
will give the invariant lines (\ref{4.3.6}).
\begin{center}
OR
\end{center}
\item $a_4=0$,\,$b_3=0$,\, $a_5=b_4$,\, $a_3=b_5$ and $(b_2-a_1)^2+4b_1a_2\ge0$. In this case, the real values of $m$ obtained from the quadratic equation 
\begin{equation*}
b_1+(b_2-a_1)m-a_2m^2=0
\end{equation*}
provide the invariant lines.
\begin{center}
OR
\end{center}
\item the coefficients in the following equations 
\begin{equation}
a_2m^2+(a_1-b_2)m-b_1=0 \label{4.3.21}
\end{equation}
and
\begin{equation}
a_4m^3+(a_5-b_4)m^2+(a_3-b_5)m-b_3=0  \label{4.3.22}
\end{equation}

are not all zero. In this case, the real values of $m$ satisfying  (\ref{4.3.21}) and (\ref{4.3.22}) simultaneously, provide the invariant lines (\ref{4.3.6}).
\begin{center}
OR
\end{center}
\item the coefficients in (\ref{4.3.21}) and (\ref{4.3.22}) are all zero. In this case, there exists infinitely many invariant straight lines (\ref{4.3.6}) for all $m\in\mathbb{R}$.
\end{enumerate}
\item The line \,$x=k$\, (respectively, $y=l$) is invariant under the flow of given system (\ref{4.3.9}), if \,$a_1k+a_2y+a_3k^2+a_4k^2+a_5ky=0
$, \, for all $y\in\mathbb{R}$ \,(respectively, $b_1x+b_2l+b_3x^2+b_4l^2+b_5xl=0$, \, for all $x\in\mathbb{R}$).
\end{enumerate}

\end{The}

\begin{remark}
The Theorem \ref{Thm 4.3.1} corresponds to the equilibrium point ${\bm 0}$ of system (\ref{4.3.2}). If $(x_0,y_0)$ is any other equilibrium, then this result can be extended to obtain invariant lines of the form $(y-y_0)=m(x-x_0)$.  
\end{remark}

\begin{Ex}\label{Ex 3.1}
Consider,
\begin{equation*}
\begin{split}
\dot{x}&=x-4x^2+2y^2+10xy\\
\dot{y}&=y+4y^2+4xy . 
\end{split}
\end{equation*}
This system satisfies the condition 1(a) of  Theorem \ref{Thm 4.3.2}.
Here, $y=0$,\, $y=x$\, and $y=-4x$ are the lines invariant  under the flow of this system.
\end{Ex}

\begin{Ex}\label{Ex 3.2}
Consider,
\begin{equation*}
\begin{split}
\dot{x}&=2x^2\\
\dot{y}&=-3x^2+y^2 . 
\end{split}
\end{equation*}
This system satisfies the conditions 1(a) and 2 of Theorem \ref{Thm 4.3.2} and the invariant lines are given by
$x=0$,\, $y=3x$\, and $y=-x$.
\end{Ex}

\begin{Ex} \label{Ex 3.3}
Consider a planar quadratic system,
\begin{equation*}
\begin{split}
\dot{x}&=-x+y-x^2+3xy\\
\dot{y}&=8x+y+3y^2-xy . 
\end{split}
\end{equation*}
It can be checked that, the condition 1(b) in Theorem \ref{Thm 4.3.2} is satisfied by this system.\\
$\therefore$ $y=4x$,\, and $y=-2x$ are invariant lines.
\end{Ex}
 
\begin{Ex} \label{Ex 3.4}
Now, consider the system
\begin{equation*}
\begin{split}
\dot{x}&=3x-y-6x^2+y^2+2xy\\
\dot{y}&=6x-2y-18x^2+4y^2+3xy 
\end{split}
\end{equation*}
satisfying 1(c) of Theorem \ref{Thm 4.3.2}.\\
The lines, $y=2x$,\, and $y=3x$ are invariant.
\end{Ex}

\begin{Ex} \label{Ex 3.5}
Here we consider the planar quadratic system
\begin{equation*}
\begin{split}
\dot{x}&=3x-6x^2+2xy\\
\dot{y}&=3y+2y^2-6xy 
\end{split}
\end{equation*}
satisfying the condition 1(d) of Theorem \ref{Thm 4.3.2}. For this system the lines $y=mx$ are invariant for all $m\in\mathbb{R}$.
\end{Ex}
In the Figure \ref{Fig 1}, we sketch vector fields for the systems given in the Examples \ref{Ex 3.1} and \ref{Ex 3.5}.
 \begin{figure}[h]
              \subfloat[]{\includegraphics[width=0.45\textwidth]{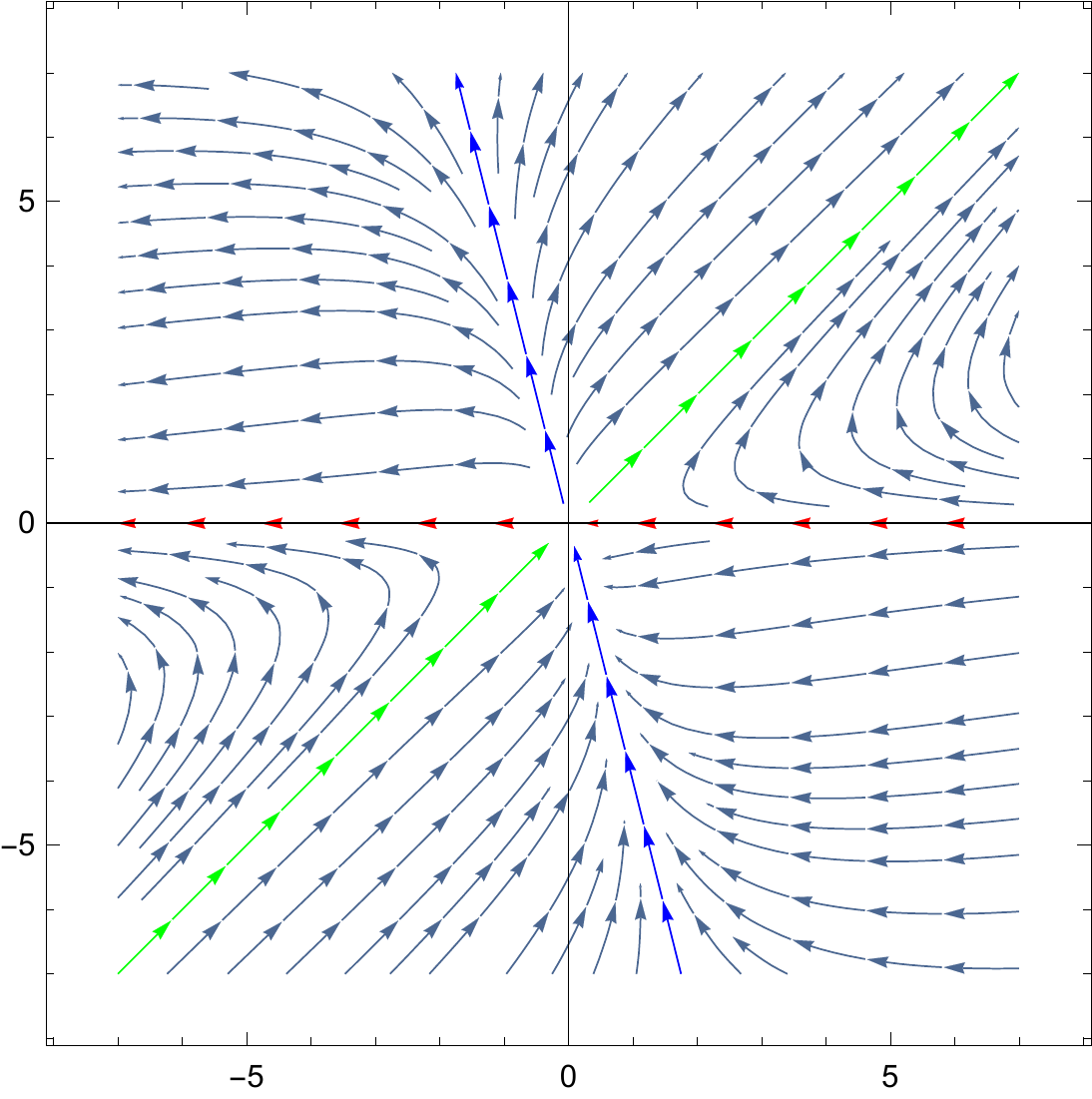}}
              \hfill 
              \subfloat[]{\includegraphics[width=0.46\textwidth]{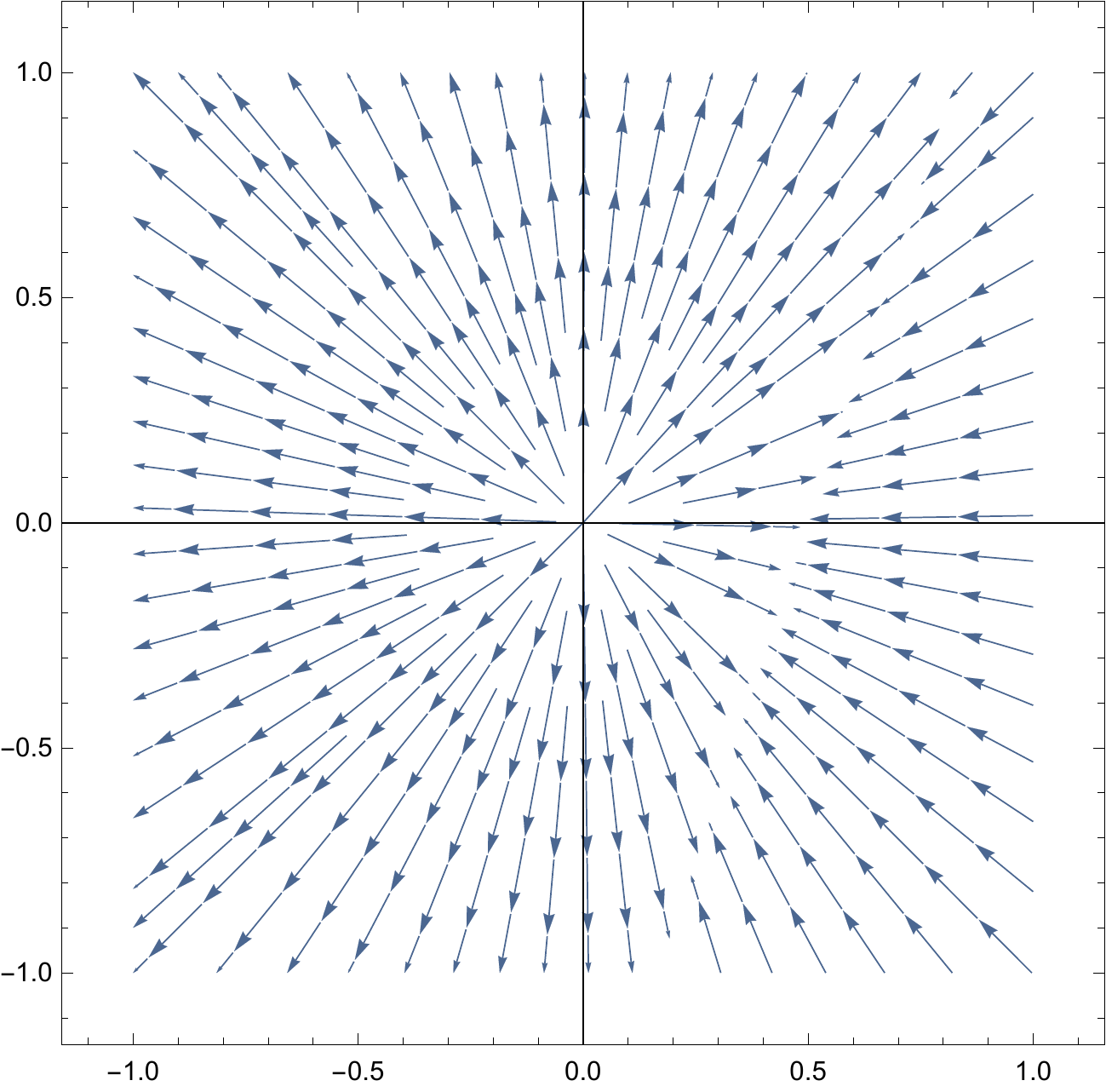}}
              \caption{Vector fields of the systems in Example 3.1 and 3.5 respectively
              } \label{Fig 1}
            \end{figure}

\begin{Ex}
Consider a planar quadratic system,
\begin{equation}
\begin{split}
\dot{x}&=2x^2-4x+2\\
\dot{y}&=6x-2y-3x^2+y^2-2. \label{4.5}
\end{split}
\end{equation}
Here, $x=1$, $y=3x-2$ and $y=-x+2$ are invariant under the flow of given system (see the Figure \ref{Fig 2}). 
\end{Ex}
\begin{figure}[h]
   \begin{center}
             \includegraphics[width=0.4\textwidth]{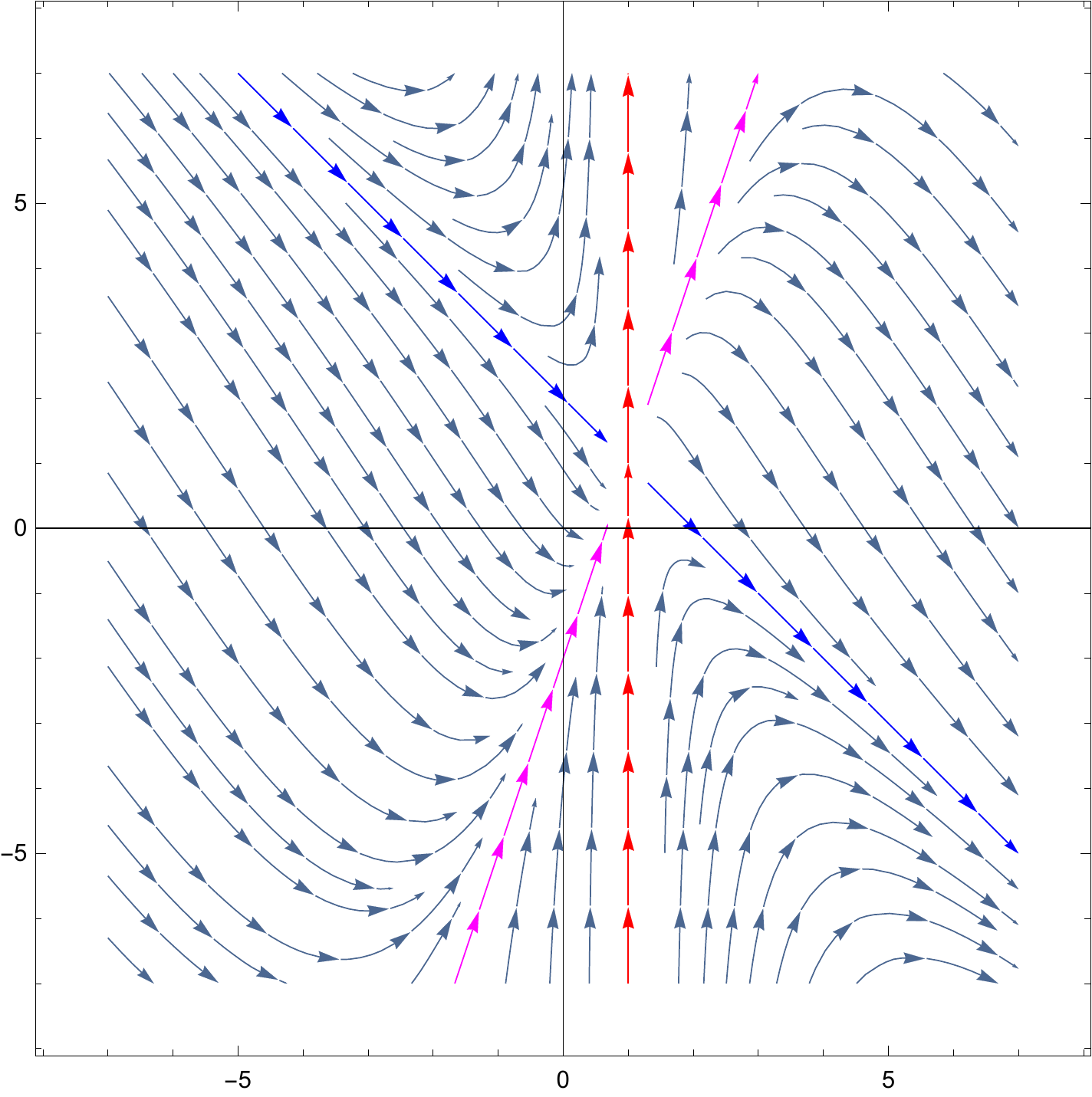}
            \caption{Vector field of system (\ref{4.5}).}         
   \label{Fig 2}
   \end{center}  
          \end{figure}

\subsection{Invariant parabolas}

\begin{The} \label{Thm 4.3.3}
Consider the planar quadratic system (\ref{4.3.9}).
The parabola $\bm {y =mx^2}$ is invariant if and only if,
\begin{enumerate}
\item $b_1=0,\,b_4=2a_5,\,a_4=0$ and
\item One of the following conditions $(a),\, (b)$, $(c)$ and $(d)$ hold:
\begin{enumerate}
\item $b_3=0$,\, $b_2=2a_1$,\, $b_5\ne2a_3$\, and\, $a_2\ne0$.\, (In this case $m=\frac{b_5-2a_3}{2a_2}$). 
\item $b_5=2a_3$,\, $a_2=0$,\, $b_3\ne0$\, and\, $b_2\ne2a_1$.\, (In this case $m=\frac{-b_3}{b_2-2a_1}$).
\item $b_3\ne0$,\, $b_2\ne2a_1$,\, $b_5\ne2a_3$,\, $a_2\ne0$ \,and\, $2a_2b_3+b_2b_5-2b_2a_3-2a_1b_5+4a_1a_3=0$.\, (In this case $m=\frac{-b_3}{b_2-2a_1}=\frac{b_5-2a_3}{2a_2}$).
\item $b_3=0$,\, $b_2=2a_1$,\, $b_5=2a_3$\,and\, $a_2=0$. (In this case $m$ is any real number).
\end{enumerate}
\end{enumerate}
\end{The}

\begin{proof}
The tangency condition shows that the parabola
 \begin{equation}
y=mx^2 \label{4.4.22}
\end{equation}
is invariant if and only if 
\begin{equation*}
b_1x+(b_2m+b_3-2a_1m)x^2+(b_5-2a_2m-2a_3)mx^3+(b_4-2a_5)m^2x^4-2a_4m^3x^5=0\quad \forall x\in\mathbb{R}.
\end{equation*} 
This holds if and only if, 
\begin{equation}
b_1=0, \,b_4=2a_5,\,a_4=0, \label{4.4.23}
\end{equation}
\begin{equation}
b_2m+b_3-2a_1m=0 \label{4.4.24}
\end{equation}
and 
\begin{equation}
b_5-2a_2m-2a_3=0. \label{4.4.25}
\end{equation}
From (\ref{4.4.24}), we have 
\begin{equation}
m=\frac{-b_3}{b_2-2a_1} \label{4.4.26}
\end{equation}
and from (\ref{4.4.25}), we have 
\begin{equation}
m=\frac{b_5-2a_3}{2a_2}. \label{4.4.27}
\end{equation}
Therefore, the parabola (\ref{4.4.22}) is invariant under the flow of system (\ref{4.3.9}) if and only if one of the conditions $(a),\, (b)$, $(c)$ and $(d)$ hold along with the condition (\ref{4.4.23}).
\end{proof}

In the Table \ref{Tab1}, we provide examples supporting the Theorem \ref{Thm 4.3.3}.

\begin{table}[h] 
\centering 
\renewcommand{\arraystretch}{1.3}
\begin{tabular}{|c |c| c| c|}
    \hline
   Ex. no. & Planar quadratic system & Related condition in Theorem \ref{Thm 4.3.3} & Invariant parabolas \\
    \hline
 $i$&  $\begin{array} {lcl} \dot{x}&=&-2x+y+3x^2-xy \\ \dot{y}&=&-4y-2y^2+5xy \end{array}$ &  (1) and 2(a) & $y=-\frac{1}{2}x^2$  \\
    \hline
$ii$ &  $\begin{array} {lcl} \dot{x}&=&3x-x^2+2xy \\ \dot{y}&=&-2y+5x^2+4y^2-2xy \end{array}$ &  (1) and 2(b) & $y=\frac{5}{8}x^2$ \\
     \hline
$iii$ &  $\begin{array} {lcl} \dot{x}&=&x-y+2x^2+xy \\ \dot{y}&=&y+x^2+2y^2+2xy \end{array}$ & (1) and 2(c)  & $y=x^2$ \\
    \hline
$iv$ &  $\begin{array} {lcl} \dot{x}&=&-x+2x^2-3xy \\ \dot{y}&=&-2y-6y^2+4xy \end{array}$ & (1) and 2(d) &  $y=mx^2$, $\forall$ $m\in\mathbb{R}$ \\
    \hline
    \end{tabular}
    \caption{Examples supporting to the Theorem \ref{Thm 4.3.3}.}
    \label{Tab1}
    \end{table}  
In the Figure \ref{Fig 3}, we sketch vector fields for the systems given in the examples $(iii)$ and $(iv)$.
   \begin{figure}[h]
                 \subfloat[$y=x^2$ is invariant parabola]{\includegraphics[width=0.46\textwidth]{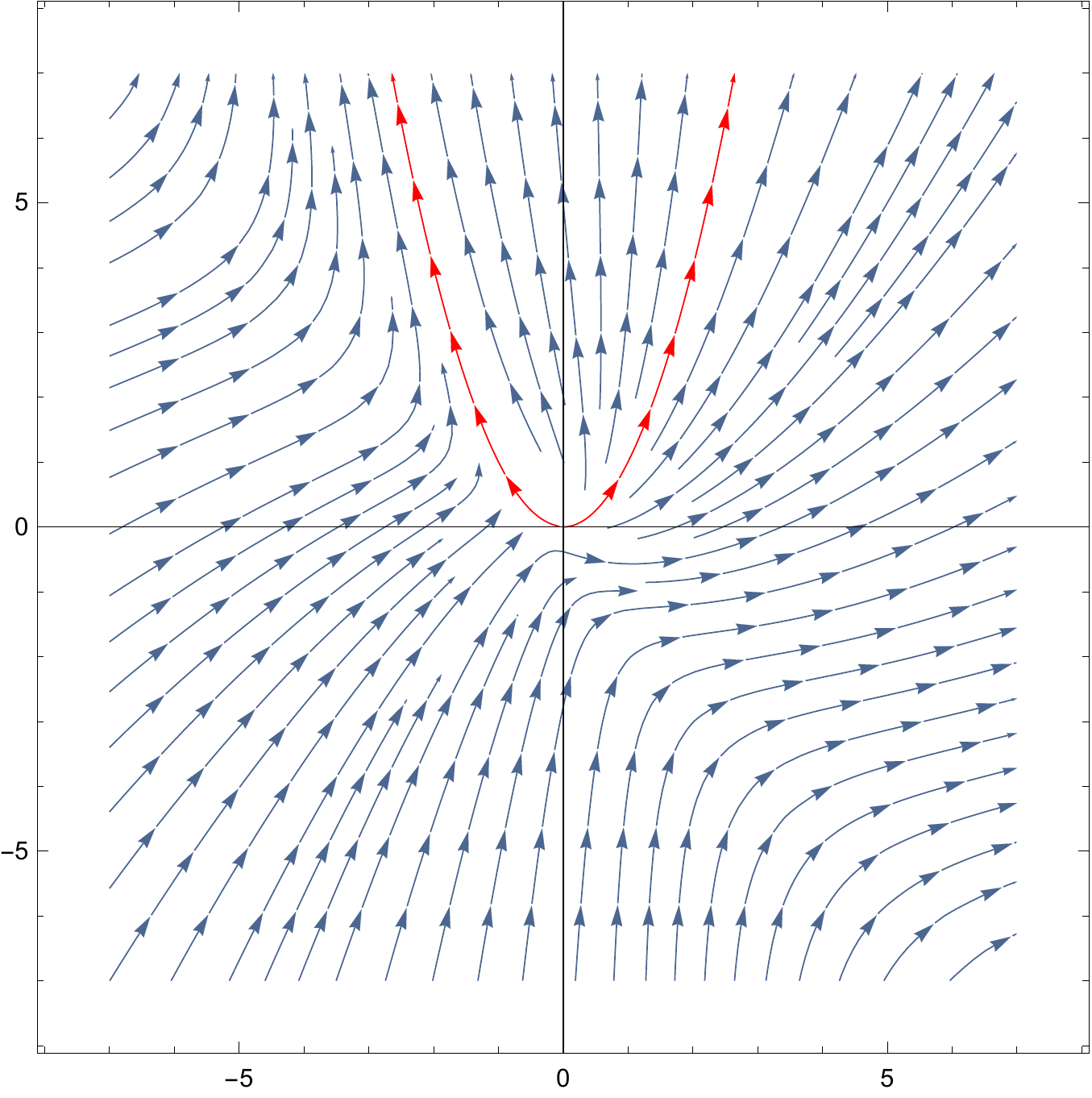}}
                 \hfill 
                 \subfloat[$y=mx^2$, $\forall m\in\mathbb{R}$ are invariant parabolas]{\includegraphics[width=0.46\textwidth]{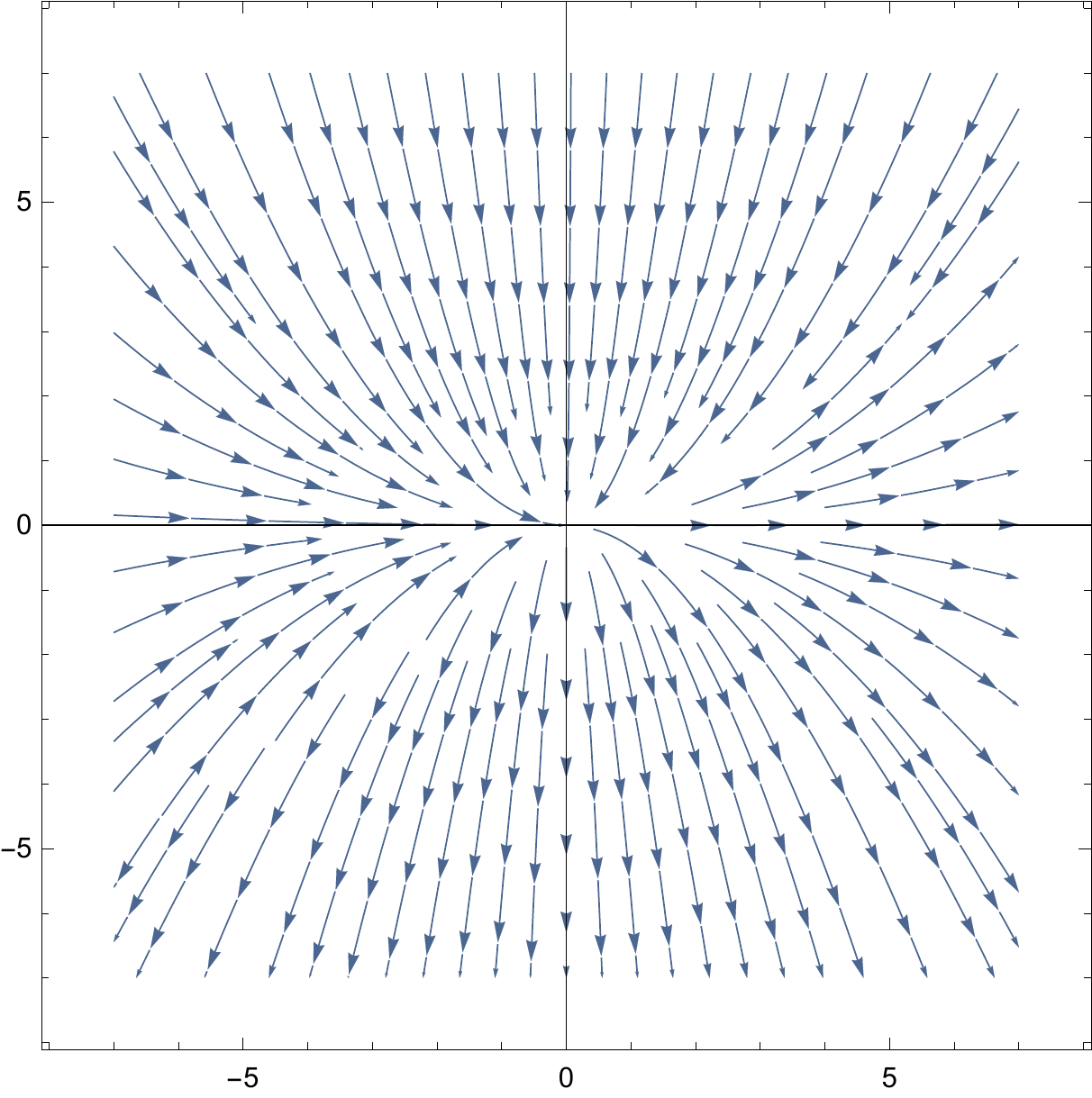}}
                 \caption{Invariant parabolas $y=mx^2$
                 } \label{Fig 3}
               \end{figure} 

\begin{The}\label{Thm 4.3.4}
There exists invariant parabola $\bm {x =my^2}$ for the system (\ref{4.3.9}) if and only if
\begin{enumerate}
\item $a_2=0,\,a_3=2b_5,\,b_3=0$ \, and
\item One of the conditions $(a),\, (b)$,\, $(c)$ and $(d)$ hold.
\begin{enumerate}
\item $a_4=0$,\, $a_1=2b_2$,\, $a_5\ne2b_4$\, and\, $b_1\ne0$.\, (In this case $m=\frac{a_5-2b_4}{2b_1}$). 
\item $a_5=2b_4$,\, $b_1=0$,\, $a_4\ne0$\, and\, $a_1\ne2b_2$.\, (In this case $m=\frac{-a_4}{a_1-2b_2}$).
\item $a_4\ne0$,\, $a_1\ne2b_2$,\, $a_5\ne2b_4$,\, $b_1\ne0$ \,and\, $a_1a_5-2a_1b_4-2b_2a_5+4b_2b_4+2a_4b_1=0$.\, (In this case $m=\frac{a_5-2b_4}{2b_1}=\frac{-a_4}{a_1-2b_2}$).
\item $a_4=0$,\, $a_1=2b_2$,\, $a_5=2b_4$\, and $b_1=0$. (In this case, $x=my^2$,\, $\forall\, m\in\mathbb{R}$).
\end{enumerate}
\end{enumerate}
\end{The}

In the Table \ref{Tab2}, we provide examples supporting to the Theorem \ref{Thm 4.3.4}.
The corresponding vector fields are sketched in Figure \ref{Fig 4}.

\begin{table}[h] 
\centering 
\renewcommand{\arraystretch}{1.3}
\begin{tabular}{|c |c| c| c|}
    \hline
   Ex. no. & Planar quadratic system & Related condition in Theorem \ref{Thm 4.3.4} & Invariant parabolas \\
    \hline
 $v$&  $\begin{array} {lcl} \dot{x}&=&-2x+6x^2-2xy \\ \dot{y}&=&4x-y+2y^2+3xy \end{array}$ &  (1) and 2(a) & $x=-\frac{3}{4}y^2$  \\
    \hline
$vi$ &  $\begin{array} {lcl} \dot{x}&=&2x-2x^2+y^2+6xy \\ \dot{y}&=&-y+3y^2-xy \end{array}$ &  (1) and 2(b) & $x=-\frac{1}{4}y^2$ \\
     \hline
$vii$ &  $\begin{array} {lcl} \dot{x}&=&2x-2x^2+y^2-xy \\ \dot{y}&=&-7x+2y+3y^2-xy \end{array}$ & (1) and 2(c)  & $x=\frac{1}{2}y^2$ \\
    \hline
$viii$ &  $\begin{array} {lcl} \dot{x}&=&4x-6x^2-10xy \\ \dot{y}&=&2y-5y^2-3xy \end{array}$ & (1) and 2(d) &  $x=my^2$, $\forall$ $m\in\mathbb{R}$ \\
    \hline
    \end{tabular}
    \caption{Examples supporting to the Theorem \ref{Thm 4.3.4}.}
    \label{Tab2}
    \end{table}  

 \begin{figure}[h]
                 \subfloat[vector fields of example  ($vi$)]{\includegraphics[width=0.46\textwidth]{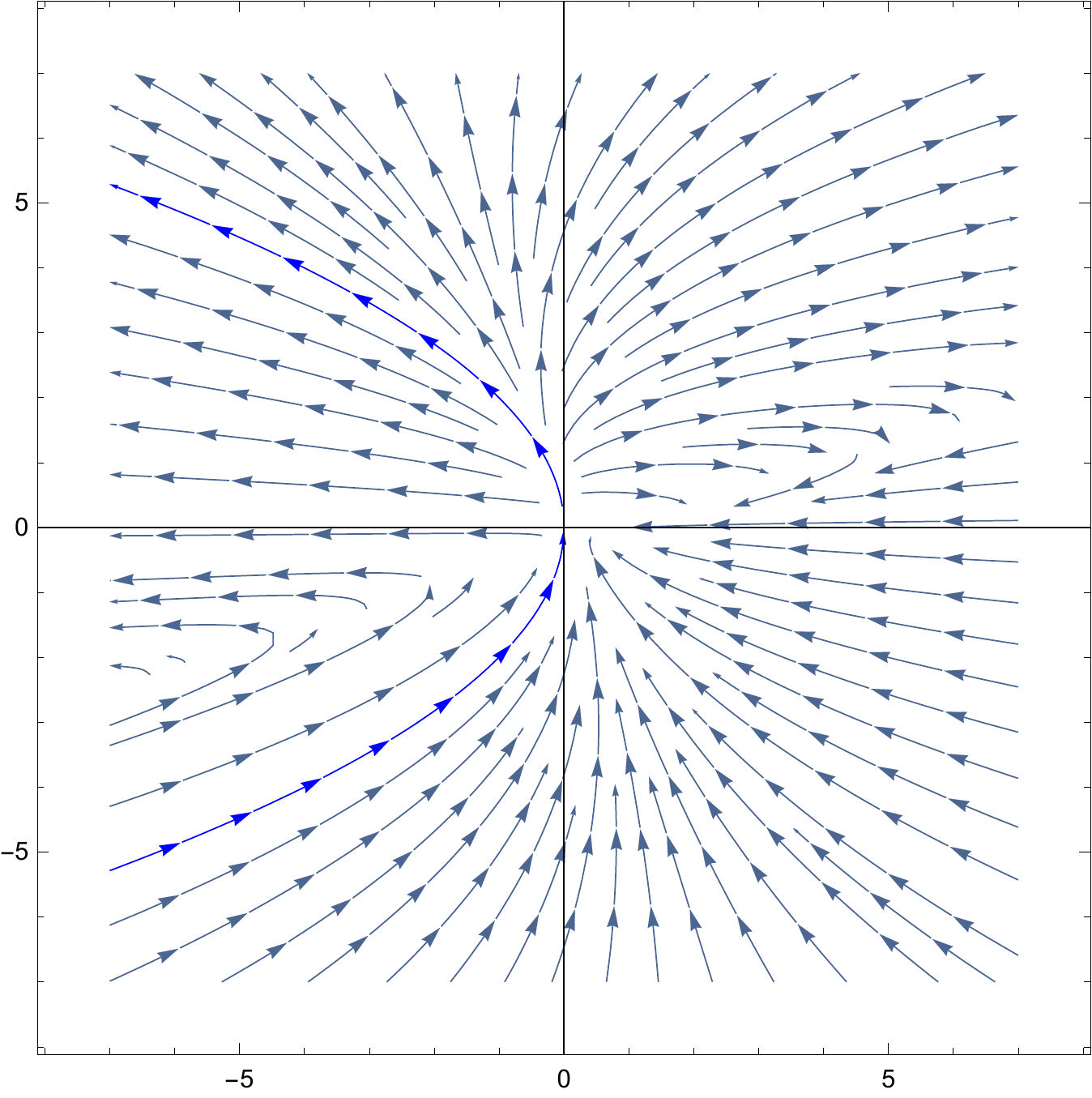}}
                 \hfill 
                 \subfloat[vector fields of example  ($viii$)]{\includegraphics[width=0.46\textwidth]{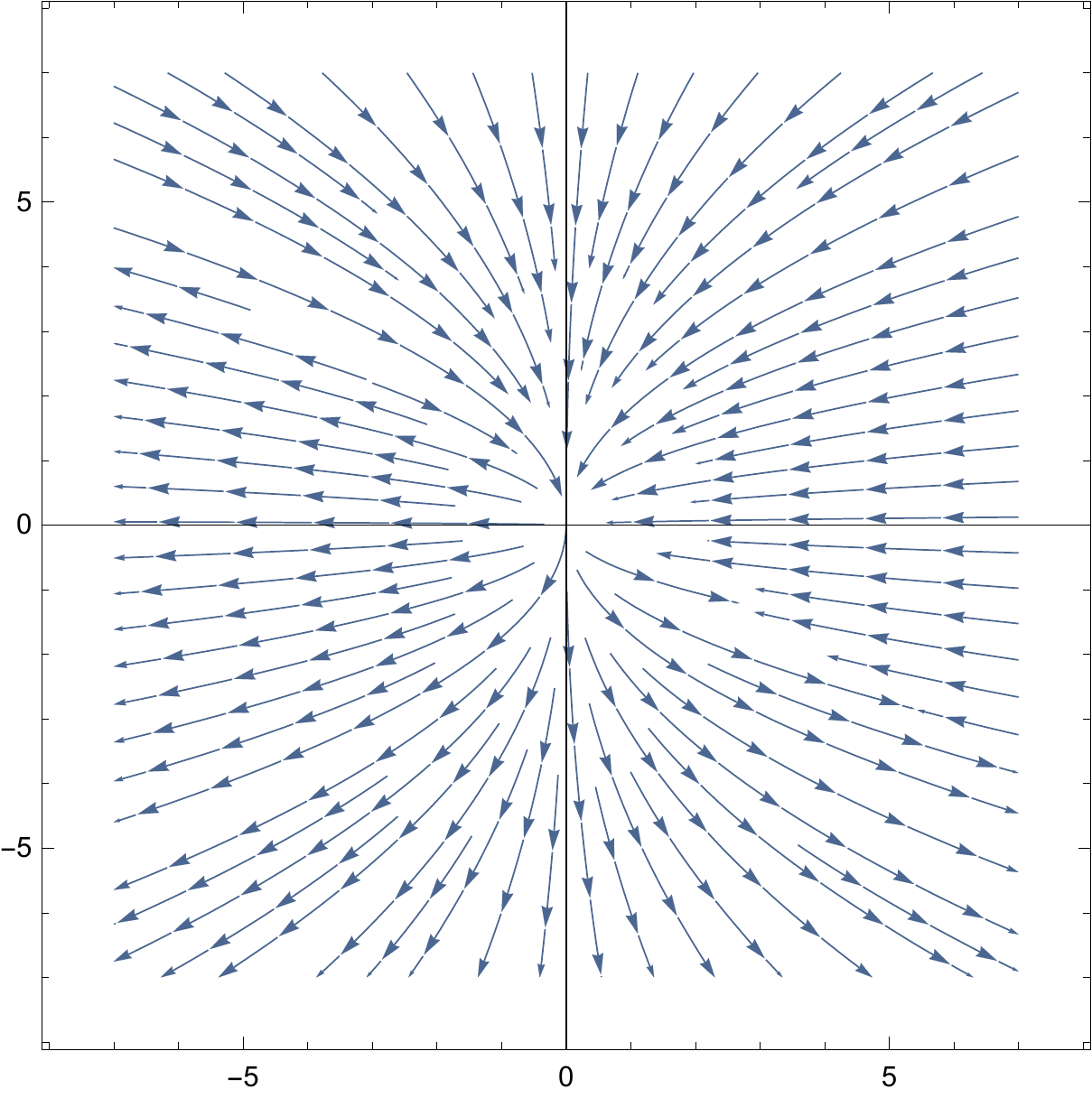}}
                 \caption{Invariant parabolas $x=my^2$
                 } \label{Fig 4}
               \end{figure}

The following Theorem \ref{Thm 4.3.5} describes the general case.
\begin{The}\label{Thm 4.3.5}
The parabola
\begin{equation}
m_1x^2+m_2xy+m_3y^2+m_4x+m_5y=0 \label{4.4.10}
\end{equation}
is invariant under the flow of system (\ref{4.3.9}), if and only if 
\begin{equation*}
\begin{split}
a_1&=b_2+b_1(-\cot\theta+\tan\theta) \quad \textrm{or} \quad \theta=\frac{1}{2}\cotinv(\frac{b_2-a_1}{b_1}),\\
a_2&=b_1,\\
a_3&=\frac{1}{16}\left[b_3 \sec\theta\,(29\sin\theta+\sin3\theta)-2b_4 \sin2\theta +2b_5(3+\cos2\theta)\right],\\
a_4&=\frac{1}{8}\tan\theta\left[b_4(3+\cos2\theta)+\tan\theta(-2b_5 \sin^2\theta+b_3 (3+\cos2\theta)\tan\theta)\right]\quad \mathrm{and}\\
a_5&=\frac{1}{16} \sec^2\theta\left[-4b_3(5+\cos2\theta)\sin^2\theta -2(-5\cos\theta+\cos3\theta)(b_4\cos\theta+b_5\sin\theta)\right].
\end{split}
\end{equation*}
The coefficients of invariant parabola (\ref{4.4.10}) are given by,
\begin{equation*}
\begin{split}
m_1&=\frac{1}{16}\sec\theta\left[b_3(3+\cos2\theta)^2+\sin2\theta(b_4\sin2\theta-b_5(3+\cos2\theta))\right],\\
m_2&=\frac{-1}{8}\sec\theta\left[b_3(3+\cos2\theta)^2+\sin2\theta(b_4\sin2\theta-b_5(3+\cos2\theta))\right]\tan\theta,\\
m_3&=\frac{1}{16}\sec\theta\left[b_3(3+\cos2\theta)^2+\sin2\theta(b_4\sin2\theta-b_5(3+\cos2\theta))\right]\tan^2\theta,\\
m_4&=-2\sin\theta(-b_1\cot\theta+b_2)+\sin\theta(b_2+b_1\tan\theta)\quad \mathrm{and}\\
m_5&=-2\cos\theta(-b_1\cot\theta+b_2)+\cos\theta(b_2+b_1\tan\theta),\, \text{where}\,\, 0\le \theta\le 2\pi.
\end{split}
\end{equation*}
\end{The}

\begin{Ex}
Consider the following system,
\begin{equation}
\begin{split}
\dot{x}&=\frac{1}{2}x-\frac{5}{2}y+2\sqrt{2}x^2+2\sqrt{2}y^2\\
\dot{y}&=-\frac{5}{2}x+\frac{1}{2}y+\frac{3}{\sqrt{2}}x^2+\frac{7}{\sqrt{2}}y^2-\sqrt{2}xy. \label{4.6}
\end{split}
\end{equation}

Here, \,$5x^2-10xy+5y^2-8\sqrt{2}x-8\sqrt{2}y=0$\, is invariant under the flow of this system (see the Figure \ref{Fig 5}).
\end{Ex}
\begin{figure}[h]
   \begin{center}
             \includegraphics[width=0.45\textwidth]{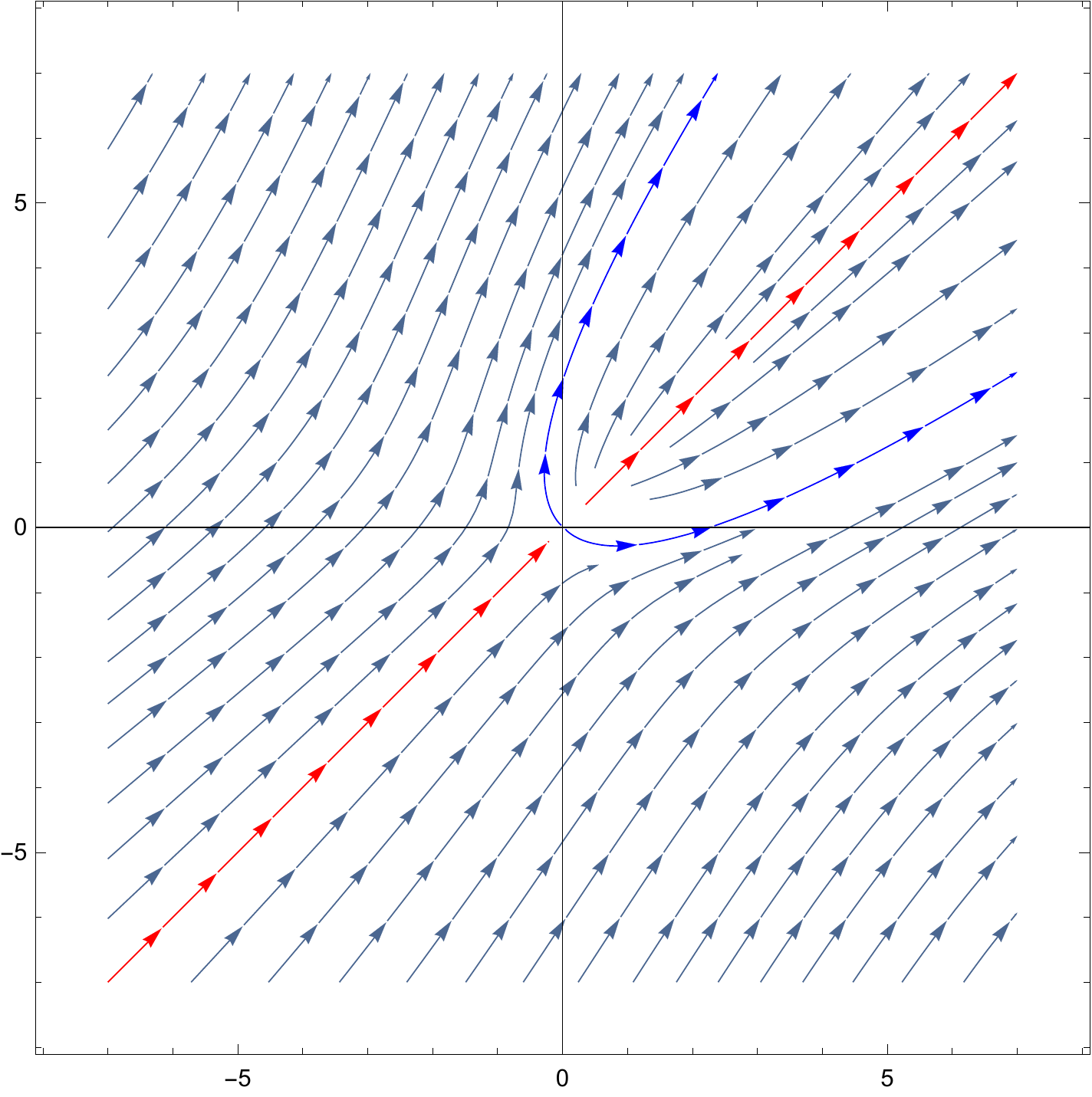}
            \caption{Vector field of system (\ref{4.6}).}         
   \label{Fig 5}
   \end{center}  
          \end{figure} 
          
 {\bf Note:}   Invariant parabolas passing through equilibrium points $(x_0,y_0)$ other than origin can be obtained using similar results.       

\subsection{Hamiltonian systems}
 The system
 \begin{equation}
\begin{split}
\dot{x}&=f(x,\,y)\\
\dot{y}&=g(x,\,y)\\ \label{4.3.31}
\end{split}
\end{equation}
 is Hamiltonian if and only if $\exists$ a function $H(x,\,y)$ \cite{Perko},  such that 
\begin{equation}
\frac{\partial H}{\partial y} =f(x,\,y)\quad
\mathrm{and}\quad \frac{\partial H}{\partial x} =-g(x,\,y). \,\, \label{4.3.32}
\end{equation}
i.e. if and only if \,\,$\frac{\partial f}{\partial x}+\frac{\partial g}{\partial y}=0$.\\
Note that, the curves 
\begin{equation}
H(x,y)=c \label{4.3.33}
\end{equation}
are invariant under system (\ref{4.3.31}), because $\frac{\mathrm{d}H}{\mathrm{d}t}=0$\,\, (by (\ref{4.3.32})) along the solution trajectories.
\par In particular, if the curve (\ref{4.3.33}) passes through a saddle equilibrium, then it is separatrix.

\begin{Ex}
Consider a planar quadratic system,
\begin{equation}
\begin{split}
\dot{x}&=y-\sqrt{2}xy\\
\dot{y}&=\frac{1}{2}(2x+\sqrt{2}x^2+\sqrt{2}y^2) . 
\end{split} \label{4.7}
\end{equation}
This system is a the Hamiltonian system and the  Hamiltonian is given by,
\begin{equation*}
H(x,\,y)=-\frac{1}{2}x^2+\frac{1}{2}y^2-\frac{1}{3\sqrt{2}}x^3-\frac{1}{\sqrt{2}}xy^2.
\end{equation*}
Here, \,$-\frac{1}{2}x^2+\frac{1}{2}y^2-\frac{1}{3\sqrt{2}}x^3-\frac{1}{\sqrt{2}}xy^2=0$\, gives separatrix for this system and it is shown in the following Figure \ref{Fig 6}.
\end{Ex}
\begin{figure}[h]
   \begin{center}
             \includegraphics[width=0.45\textwidth]{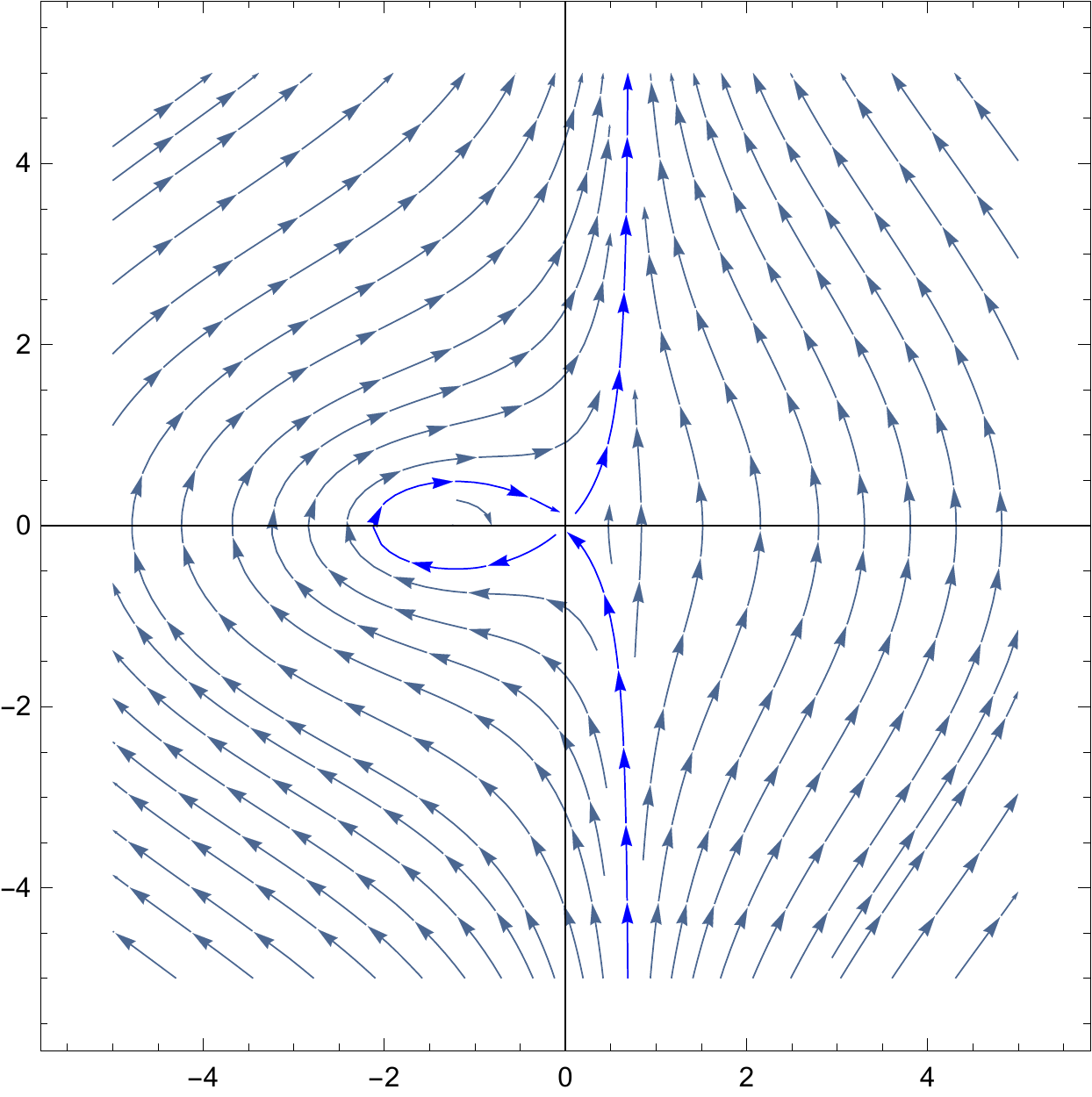}
            \caption{Vector field of system (\ref{4.7}).}         
   \label{Fig 6}
   \end{center}  
          \end{figure}
The homoclinic loop passing through the origin contains a center. It is an intersection of stable manifold
$$S=\{(x,y)\in\mathbb{R}^2:\frac{x^2}{6}-\frac{x y}{3}+\frac{y^2}{6}+\frac{x}{\sqrt{2}}+\frac{y}{\sqrt{2}}=0\}$$ and unstable manifold $$U=\{(x,y)\in\mathbb{R}^2:\frac{x^2}{6}+\frac{x y}{3}+\frac{y^2}{6}+\frac{x}{\sqrt{2}}-\frac{y}{\sqrt{2}}=0\}$$ of system (\ref{4.7}).

\subsection{Some other invariant curves}
\begin{The}
The cubic curve $y=x^3+mx^2+ux$\, is invariant under planar quadratic system (\ref{4.3.9}) if and only if 
\begin{equation*}
a_4=a_2=b_4=a_5=0,\,\, b_5=3a_3,
\end{equation*}
\begin{equation*}
6a_1^3+2a_3^2b_1-11a_1^2b_2-b_2^3-a_3b_2b_3+a_1(6b_2^2+a_3b_3)=0,
\end{equation*}
$a_3\ne0$ and $a_1\ne b_2$.\\
In this case, 
\begin{equation*}
m=\frac{3a_1-b_2}{a_3}\quad \mathrm{and}\quad u=\frac{b_1}{a_1-b_2}.
\end{equation*}
\end{The}

\begin{Ex}
The planar quadratic system
\begin{equation}
\begin{split}
\dot{x}&=x+x^2\\
\dot{y}&=x+2y+2x^2+3xy,\label{4.8} 
\end{split}
\end{equation}
has invariant curve viz. 
$y=x^3+x^2-x$ (see Figure \ref{Fig 7}).
\end{Ex}
\begin{figure}[h]
   \begin{center}
             \includegraphics[width=0.45\textwidth]{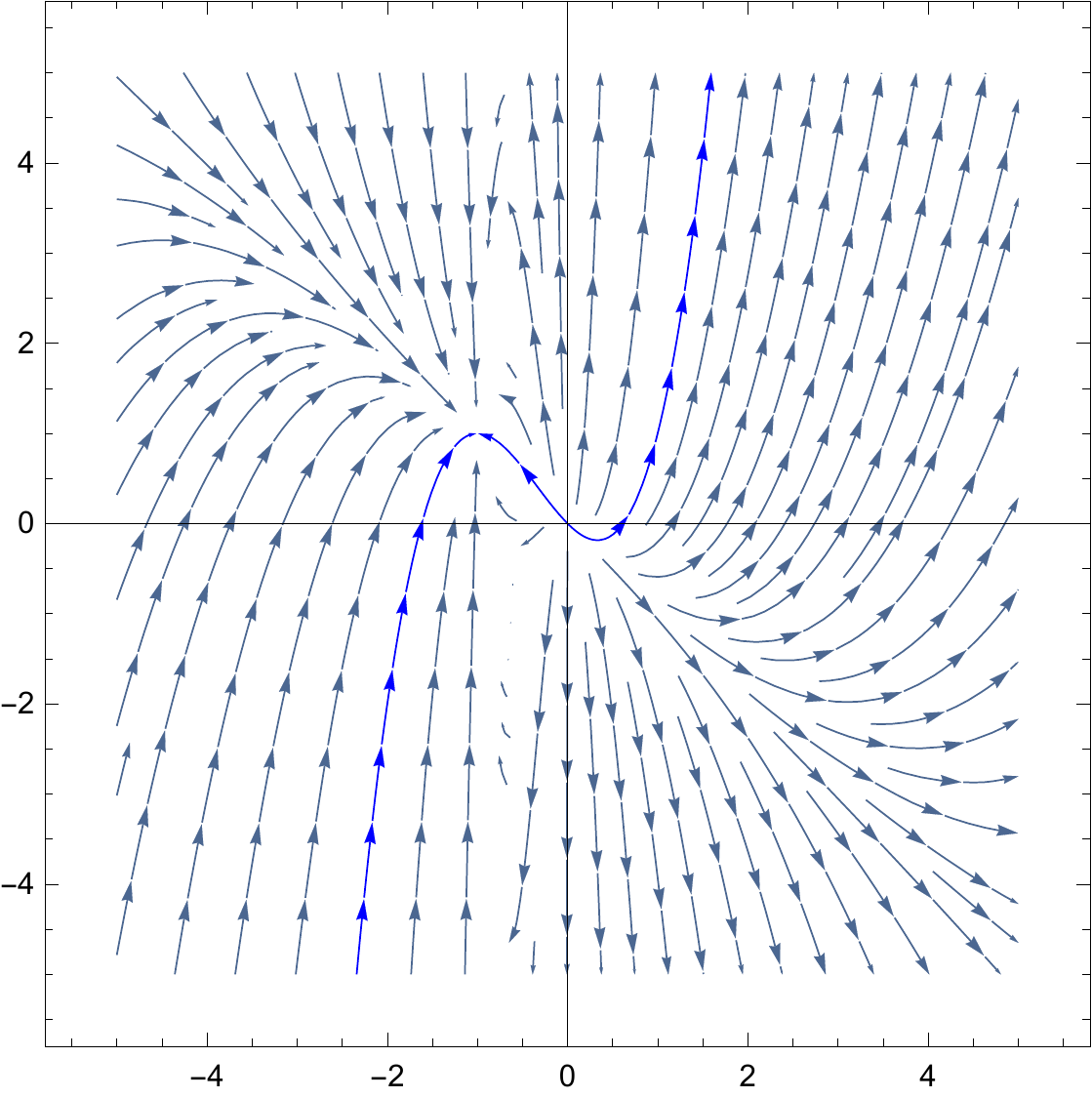}
            \caption{Vector field of system (\ref{4.8}).}         
   \label{Fig 7}
   \end{center}  
          \end{figure}

\begin{The}
The curves $y=mx^k$, (for any $m\in\mathbb{R}$ and $k>0$) are invariant under the flow of planar quadratic system (\ref{4.3.9}) if and only if   
\begin{equation*}
b_1=b_3=a_4=a_2=0,\,\, b_4=ka_5,
\end{equation*}
\begin{equation*}
b_5=ka_3\quad \mathrm{and}\quad b_2=ka_1.
\end{equation*}
\end{The}

The system (\ref{4.3.9}) can have invariant curves other than polynomial curves also. The following theorem provides conditions for the existence of exponential curve as an invariant.

\begin{The}
The system (\ref{4.3.9}) has exponential curve $y=m e^x$ for all $m\in\mathbb{R}$, as an invariant curve if
\begin{equation*}
a_3=a_4=a_5=b_1=b_2=b_3=0,
\end{equation*}
\begin{equation*}
b_4=a_2\quad \mathrm{and}\quad b_5=a_1.
\end{equation*}
\end{The}

\begin{Ex}
Consider, the planar quadratic system
\begin{equation}
\begin{split}
\dot{x}&=-2x+3y\\
\dot{y}&=3y^2-2xy. \label{4.9} 
\end{split}
\end{equation}
It can be verified that the curve $y=m e^x$ (for all $m\in\mathbb{R}$) are invariant under the flow of system (\ref{4.9}). (see Figure \ref{Fig 8}).
\end{Ex}
\begin{figure}[h]
   \begin{center}
             \includegraphics[width=0.45\textwidth]{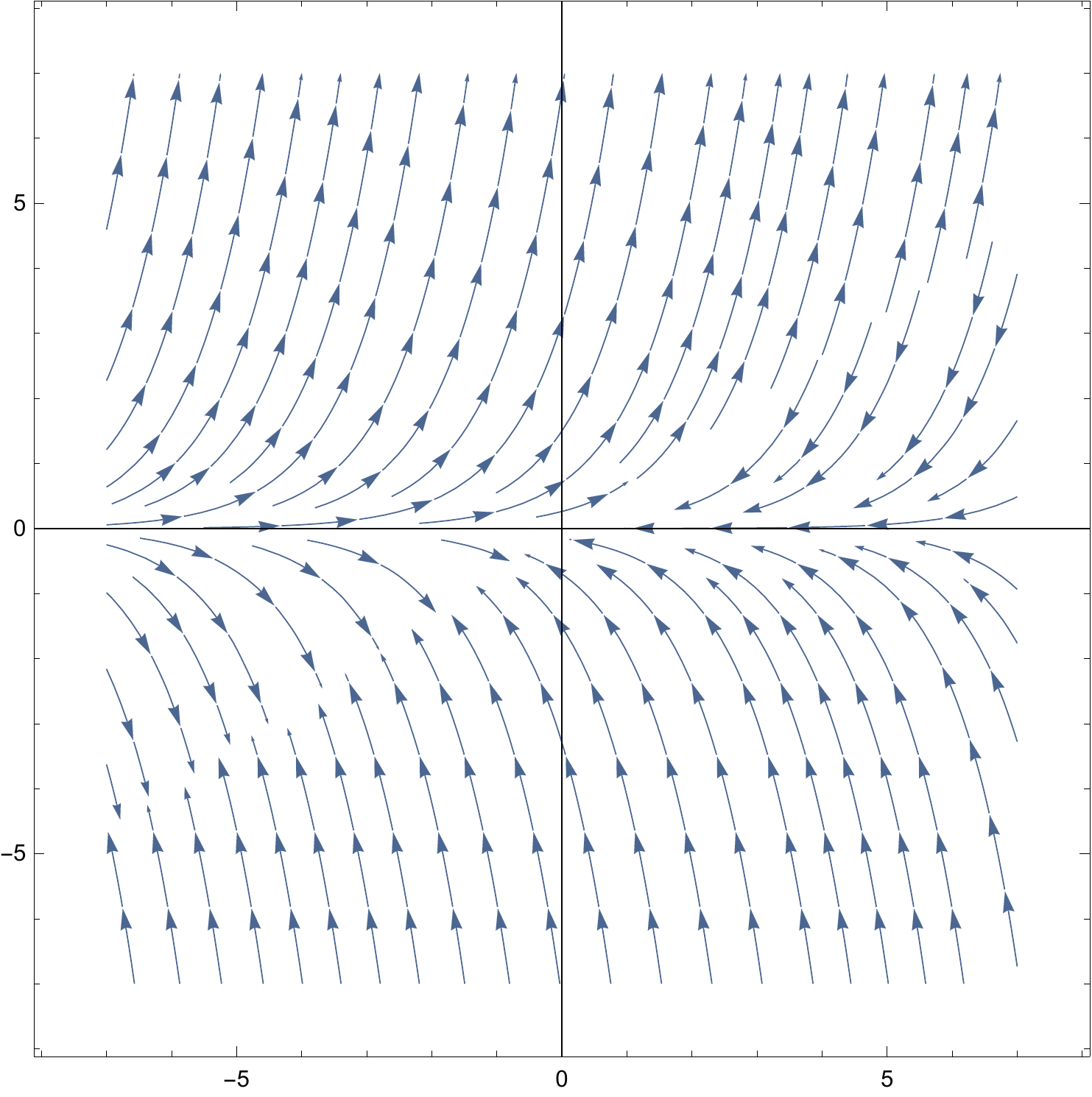}
            \caption{Vector field of system (\ref{4.9}).}         
   \label{Fig 8}
   \end{center}  
          \end{figure}

{\bf Note:}\\
Note that, the system (\ref{4.3.9}) cannot have $y=m \sin x$ as an invariant curve. In this case, the tangency condition implies that,
\begin{equation}
\begin{split}
\,& ma_1x\cos x+m^2a_2\sin x\cos x+ma_3x^2\cos x+m^3a_4\sin^2x\cos x+m^2a_5x\sin x\cos x\\ &-b_1x-b_2m\sin x-b_3x^2-b_4m^2\sin^2x-b_5mx\sin x=0, \quad \forall x\in \mathbb{R}. \label{4.3.34}
\end{split}
\end{equation}
This equation is not helpful in finding the values of $m$.\\
e.g. $x=\frac{\pi}{2}$ produces $m=\frac{-2b_2-b_5\pi\pm2\sqrt{(b_2+b_5\frac{\pi}{4})^2-4b_4(b_1\frac{\pi}{2}+b_3\frac{\pi^2}{4})}}{4b_4}$ depending on $b_i's$ only. However $x=\pi$ produces $m=\frac{-b_1\pi-b_3\pi^2}{a_1\pi+a_3\pi^2}$ depending on $a_i's$ as well as $b_i's$. Therefore we cannot find unique $m\ne0$ satisfying (\ref{4.3.34}).

\section{Fractional order systems} \label{Sec4}
The fractional order systems are generalizations of classical systems. In this section, we show that these systems cannot have invariant manifolds other than the linear subspaces of $\mathbb{R}^n$.
\subsection{Invariant subspaces of fractional order systems}
\begin{The}
The conditions for the existence of invariant linear subspaces for the fractional order systems are same as their classical counterparts.
\end{The}
\begin{proof}
Consider the fractional order system
\begin{equation}
{}_0^C\mathrm{D}_t^\alpha x_i=f_i(x_1,x_2,\dots,x_n), \quad 1\le i \le n, \quad 0<\alpha<1 \label{4.38}
\end{equation}
and its classical counterpart 
\begin{equation}
\dot{x_i}=f_i(x_1,x_2,\dots,x_n), \quad 1\le i \le n. \label{4.39}
\end{equation}
The tangency condition shows that the linear subspace 
\begin{equation}
S=\{(x_1,x_2,\dots,x_n)\in\mathbb{R}^n:\sum_{i=1}^{n}a_ix_i=c, \, a_i\in\mathbb{R}\}
\end{equation}
of $\mathbb{R}^n$ is invariant under system (\ref{4.38}) if 
\begin{equation}
\sum_{i=1}^{n}a_i \,{}_0^C\mathrm{D}_t^\alpha x_i=0
\label{4.41}
\end{equation}
and invariant under system (\ref{4.39}) if 
\begin{equation}
\sum_{i=1}^{n}a_i \,\dot{x_i}=0.
\label{4.42}
\end{equation}
$\therefore$ Both the conditions (\ref{4.41}) and (\ref{4.42}) get reduced to
\begin{equation}
\sum_{i=1}^{n}a_i \,f_i(x_1,x_2,\dots,x_n)=0.
\label{4.43}
\end{equation}
This proves theorem.
\end{proof}

This shows that the Theorems \ref{Thm 4.3.1} and \ref{Thm 4.3.2} 
 of classical system hold for fractional order system (\ref{4.38}) also.

\begin{Ex}
Consider a fractional order planar quadratic system 
\begin{equation}
\begin{split}
{}_0^C\mathrm{D}_t^\alpha \,x&=x-y+2x^2-xy\\
{}_0^C\mathrm{D}_t^\alpha\, y&=-9x+y-y^2+2xy. \label{4.4.*}
\end{split}
\end{equation}

This system satisfies the condition 1(b) of  Theorem \ref{Thm 4.3.2}.
\end{Ex}
Here, $y=3x$,\, and $y=-3x$ are invariant lines. In Figure \ref{Fig 9}, we have sketched solutions of (\ref{4.4.*}) using NPCM \cite{V. Gejji 3} for $\alpha=0.75$ with various initial conditions (shown in different colors) on these invariant lines. It can be checked that, all these solutions follow the same trajectory on the respective straight line.

\begin{figure}[h]
                 \subfloat[The line y=3x is invariant]{\includegraphics[width=0.46\textwidth]{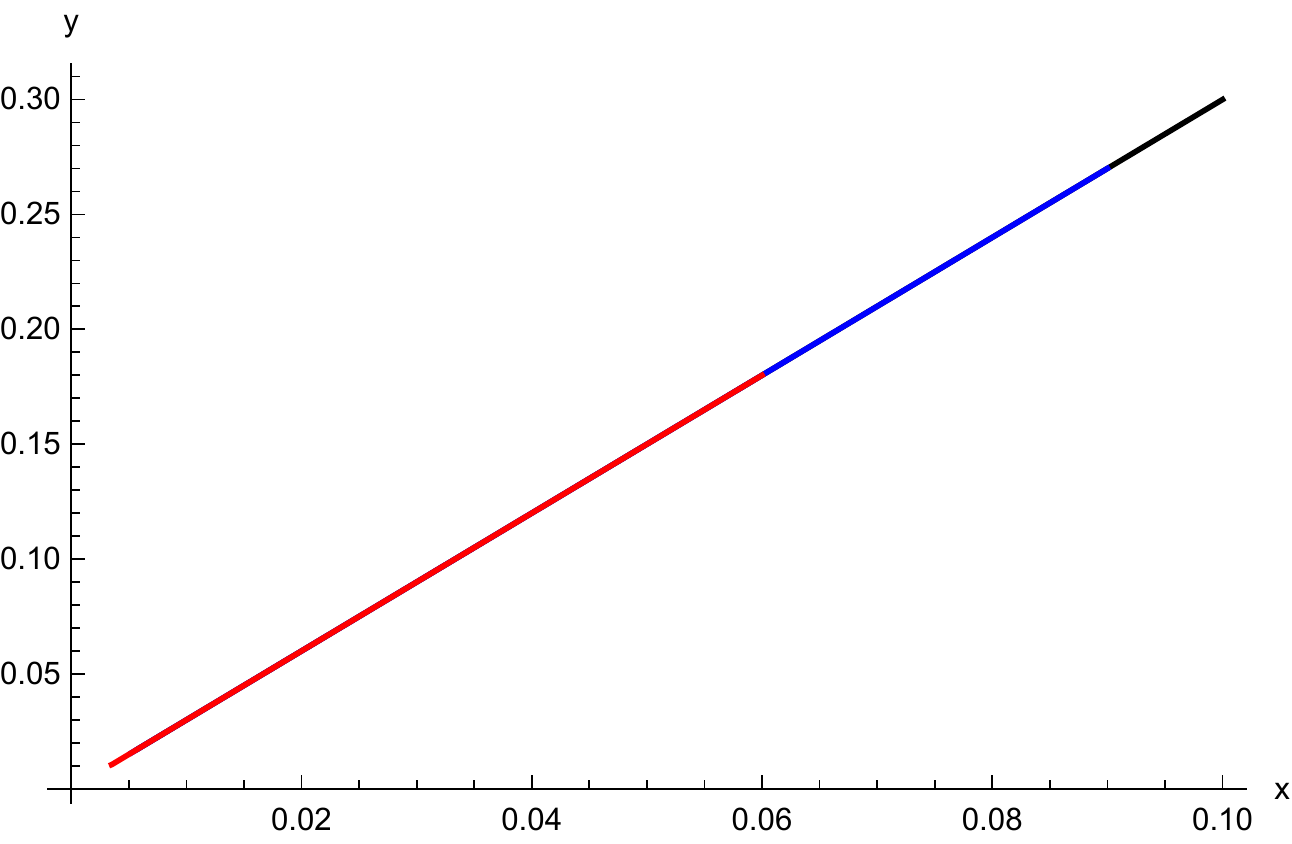}}
                 \hfill 
                 \subfloat[The line y=-3x is invariant]{\includegraphics[width=0.46\textwidth]{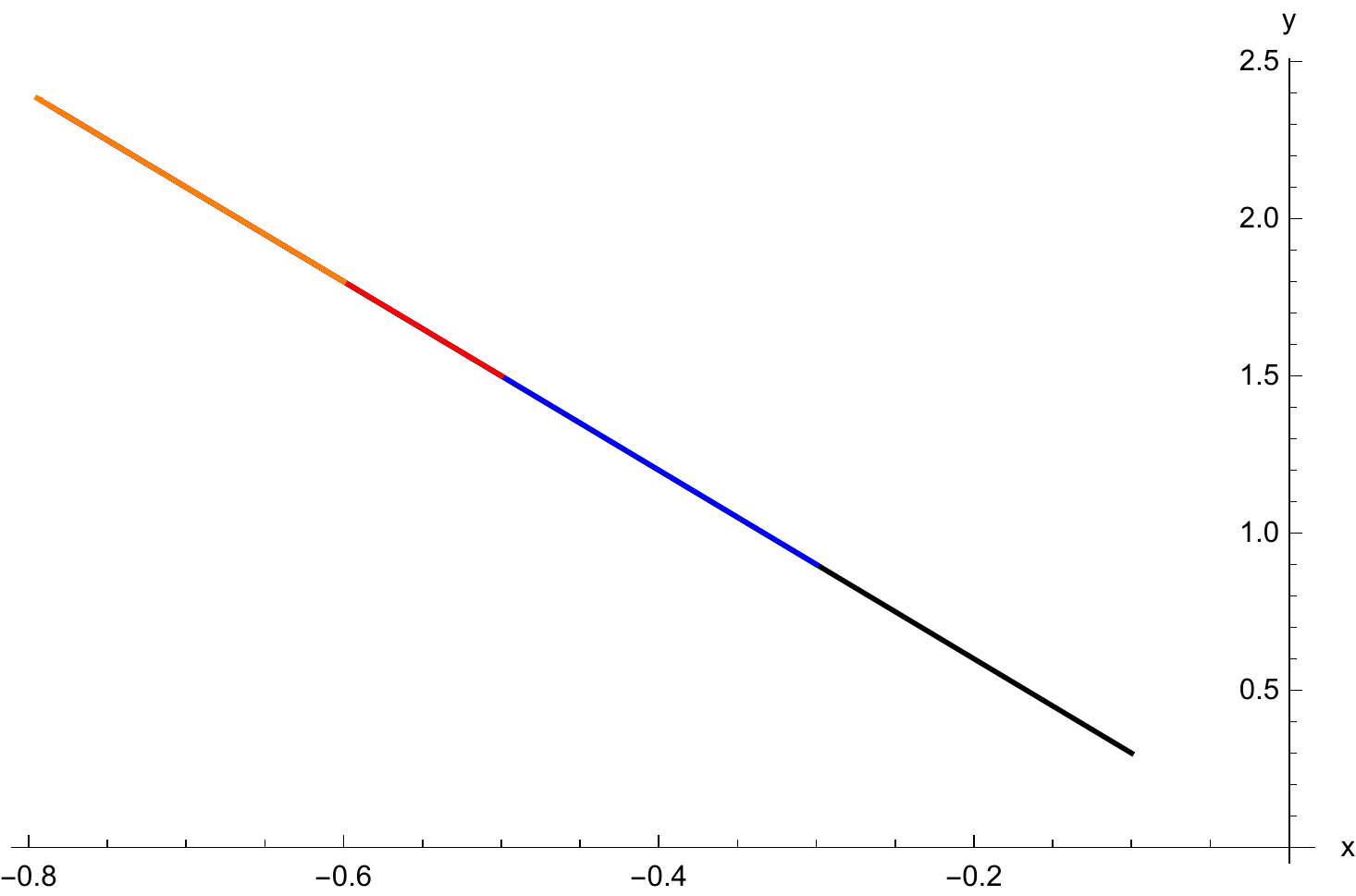}}
                 \caption{Invariant lines of system (\ref{4.4.*})
                 } \label{Fig 9}
               \end{figure} 

\subsection{Nonexistence of the invariant curves, with curvature $>0$ for fractional order systems}
\begin{The}\label{Thm 4.4.1}
The solution curves of linear FDEs 
\begin{equation}
{}_0^C\mathrm{D}_0^{\alpha}X(t)= AX(t),\,\,(0<\alpha<1) \label{4.46}
\end{equation}
for which the initial point is not on an eigenvector of $n\times n$ matrix $A$, are not invariant under $\Phi_t=E_\alpha(A t^\alpha)$.
\end{The}
\begin{proof}
The general solution of the initial value problem
\begin{equation}
{}_0^C\mathrm{D}_0^{\alpha}X(t)= AX(t),\,\,X(0)=X_0 \label{4.4.45}
\end{equation}
is 
\begin{equation}
X(t)=E_\alpha(A t^\alpha)X_0. \label{4.4.46}
\end{equation}
If $X_0$ is on an eigenvector of $A$, then $AX_0=\lambda X_0$, where $\lambda$ is the  corresponding eigenvalue.
\begin{equation}
\Rightarrow  X(t)=E_\alpha(\lambda t^\alpha)X_0. \label{4.4.47}
\end{equation}
This is on the same eigenvector, because $E_\alpha(\lambda t^\alpha)$ is a number for any $t>0$.\\
$\therefore$ The solution trajectory of (\ref{4.4.45}) starting on eigenvector is a straight line and is invariant under $\Phi_t$.
\par Now, assume that $X_0$ is not on any eigenvector of $A$.\\
$\therefore  AX_0$ is not on a vector $X_0$.\, 
In this case, if $Y_0=E_\alpha(A t_*^\alpha)X_0$, $t_*>0$ is any point on the solution curve (\ref{4.4.46}), then 
\begin{equation*}
\begin{split}
\phi_t(Y_0)&=E_\alpha(A t^\alpha)E_\alpha(A t_*^\alpha)X_0\\
&= \sum_{k=0}^{\infty} \sum_{l=0}^{k}\frac{A^k t^{\alpha k}t_*^{\alpha k-\alpha l}}{\Gamma(\alpha l+1)\Gamma(\alpha k-\alpha l+1)}X_0.
\end{split}
\end{equation*}
This cannot be written as $E_\alpha(A s^\alpha)X_0$ for any $s>0$.
\par This proves the result.
\end{proof}

\begin{Ex}
Consider the planar fractional order system
\begin{equation}
{}_0^C\mathrm{D}_0^{0.7}X(t)=\begin{bmatrix}
1 & 3\\
-3 & 1
\end{bmatrix}
X(t), \quad X(0)=X_0. \label{4.44}
\end{equation}
Its solution with $X_0=[1, \,1]^T$ is given by,
\begin{equation}
X(t)=\begin{bmatrix}
Re[E_{0.7} ((1+3i)t^\alpha)] + Im[E_{0.7} ((1+3i)t^\alpha)]\\
-Im[E_{0.7} ((1+3i)t^\alpha)] + Re[E_{0.7} ((1+3i)t^\alpha)]
\end{bmatrix}\label{4.45}
.
\end{equation}
In the Figure \ref{Fig 10}, we sketch the solution trajectory (\ref{4.45})  (Blue color) of system (\ref{4.44}) and another solution trajectories  with initial conditions at various points  $X(t_1)$ on (\ref{4.45}). It can be checked that, the trajectories follow different paths.
\end{Ex}
\begin{figure}[h]
   \begin{center}
             \includegraphics[width=0.6\textwidth]{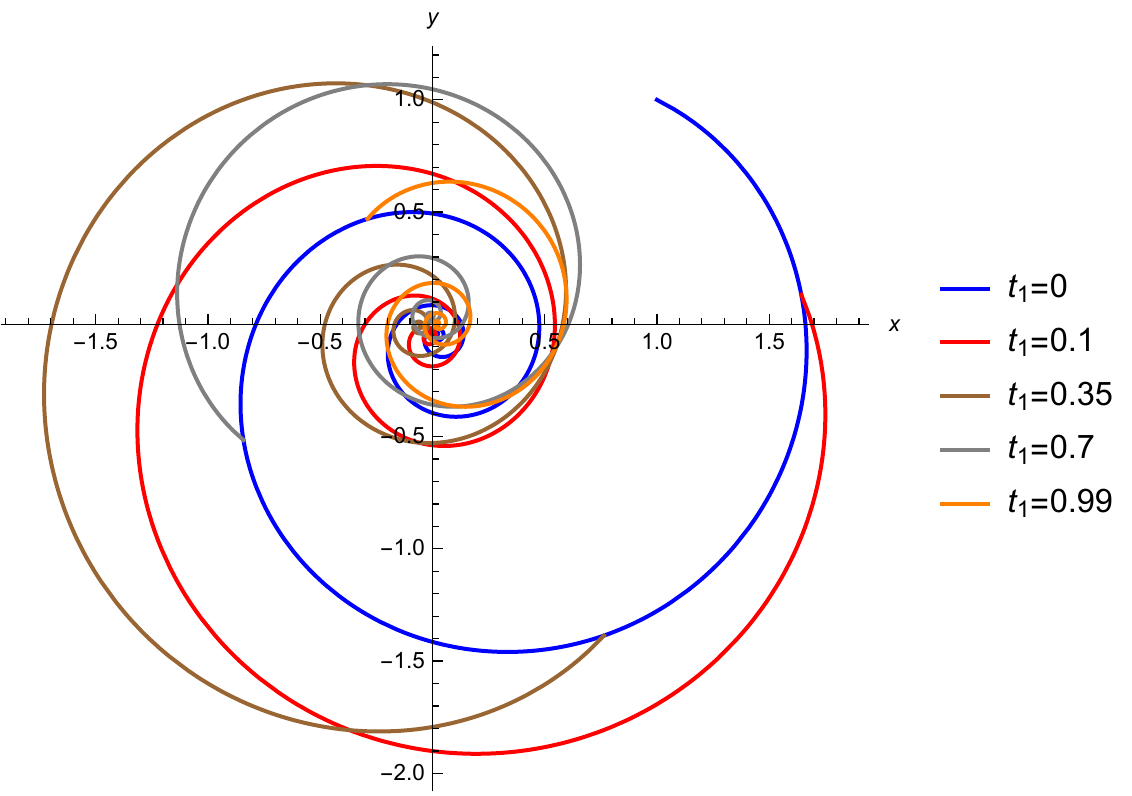}
            \caption{Solution curve  (\ref{4.45}) is not invariant.}         
   \label{Fig 10}
   \end{center}  
          \end{figure}

{\bf Note:}
\par It can be easily checked that, the tangency condition used for nonlinear case in classical sense will not provide any invariant curves for fractional order case.\\
e.g. as in Theorem  \ref{Thm 4.3.3}, consider fractional order system 
\begin{equation}
\begin{split}
{}_0^C\mathrm{D}_t^\alpha \,x&=a_1x+a_2y+a_3x^2+a_4y^2+a_5xy\\
{}_0^C\mathrm{D}_t^\alpha\, y&=b_1x+b_2y+b_3x^2+b_4y^2+b_5xy, 
\end{split} \label{4.4.1}
\end{equation}
 and a parabola $y=mx^2$.\\
Operating ${}_0^C\mathrm{D}_t^\alpha$ on both the sides, we get
\begin{equation}
{}_0^C\mathrm{D}_t^\alpha y=m\, {}_0^C\mathrm{D}_t^\alpha x^2 \label{4.4.3}
\end{equation}
Note that, unlike in classical case, ${}_0^C\mathrm{D}_t^\alpha x^2$ cannot be written in terms of ${}_0^C\mathrm{D}_t^\alpha x$. The generalized Leibniz rule \cite{Podlubny} gives
\begin{equation}
{}_0^C\mathrm{D}_t^\alpha x(t)^2=\left[{}_0^C\mathrm{D}_t^\alpha x(t)+\frac{t^{-\alpha}}{\Gamma(1-\alpha)}x(0)\right]x(t)
+\sum_{k=1}^{\infty}\binom{\alpha}{k}\left({}_0\mathrm{I}_t^{k-\alpha} x(t)\right)x^{(k)}(t)
-\frac{t^{-\alpha}}{\Gamma(1-\alpha)}x(0)^2.
\end{equation}
$\therefore$ (\ref{4.4.3}) becomes,
\begin{equation}
\begin{split}
& a_4m^3x^5+(a_5-b_4)m^2x^4+(a_2m+a_3-b_5)mx^3+(a_1m-b_2m-b_3)x^2+\left(\frac{m x(0)t^{-\alpha}}{\Gamma(1-\alpha)}-b_1\right)x\\
&+m\sum_{k=1}^{\infty}\binom{\alpha}{k}\left({}_0\mathrm{I}_t^{k-\alpha} x(t)\right)x^{(k)}(t)-\frac{m x(0)^2t^{-\alpha}}{\Gamma(1-\alpha)}=0\quad \forall t\in \mathbb{R}.
\end{split}\label{4.49}
\end{equation}
This does not provide any nonzero value of $m$, because of the term $m\sum_{k=1}^{\infty}\binom{\alpha}{k}\left({}_0\mathrm{I}_t^{k-\alpha} x(t)\right)x^{(k)}(t)$ involved in (\ref{4.49}).\\
$\Rightarrow \nexists$ any invariant parabola of the form $y=m x^2$ for fractional order system (\ref{4.4.1}). The similar computations can be used to show that $\nexists$ any invariant manifold (except linear subspace of $\mathbb{R}^n$) for fractional order system (\ref{4.4.1}).

\begin{The} \label{Thm 4.4.2}
Consider fractional order system
\begin{equation}
{}_0^C\mathrm{D}_t^\alpha X=f(X), \qquad 0<\alpha<1 \label{4.4.2}
\end{equation}
where, $f\in C^1(E)$ and $E$ is an open set in $\mathbb{R}^n$.\\
Suppose $X_*=0$ is an equilibrium of (\ref{4.4.2}).
\\
Let $\Phi_t(X_0)$ be the solution of (\ref{4.4.2}) with initial condition $X(0)=X_0$.\\
Any curve $g(x,\,y)=c$ \,\,which is not a straight line (i.e. curvature $>0$) cannot be invariant under $\Phi_t$.
\end{The}
\begin{proof}
If \,$S : \,g(x,\,y)=c$\, is an invariant set under $\Phi_t$ then, 
``any solution of (\ref{4.4.2}) starting on $S$ will stay on $S$ for all the time".\\
Let $X(t)$ be the solution of (\ref{4.4.2}) with $X(0)=X_0$ on the curve $g(x,\,y)=c$.\\
Let $Y_0=X(t_*)$, $t_*>0$ be any point on this solution curve i.e. on $g(x,\,y)=c$.
\par If the solution of (\ref{4.4.2}) starting at $Y_0$ follows the same path on $g(x,\,y)=c$, then it contradicts Theorem \ref{Thm 4.4.1}, because the linear system (\ref{4.46}) is a particular case of (\ref{4.4.2}).\\
$\therefore$  $g(x,\,y)=c$ cannot be invariant under $\Phi_t$, if $g$ is not a straight line.
\end{proof}

The generalization of Theorem \ref{Thm 4.4.2} is as below:
\begin{The}\label{Thm 4.4.3}
The fractional order system (\ref{4.4.2}) cannot have invariant manifolds other than linear subspaces of $\mathbb{R}^n$.
\end{The}

\section{Comments on the invariant manifolds in fractional order systems presented in the literature} \label{Sec5}
It is clear from the discussion in Section \ref{Sec4} that the literature \cite{Cong,Sayevand,Deshpande,Deshpande1,Cong1,Ma,Wang} developed for the local invariant manifolds in fractional order systems cannot provide correct results. In fact, it is not verified in any of these papers that whether the invariant manifolds $S$ obtained are satisfying following properties: \\
If $X_0$ is any initial condition, sufficiently close to equilibrium $X_*$, then the solution $\Phi_t(X_0)$ of given system starting at $X_0$
\begin{enumerate}
\item converge to $X_*$ as $t\rightarrow\infty$ if $S$ is stable manifold and as $t\rightarrow-\infty$ if $S$ is unstable manifold
\item stay on $S$ for all the time.
\end{enumerate}
e.g. The local stable manifold given in the paper \cite{Cong} does not follow this property as explained below:
\begin{Ex}
Consider 
\begin{equation}
\begin{split}
{}_0^C\mathrm{D}_t^\alpha x & = x-y^2,\\
{}_0^C\mathrm{D}_t^\alpha y & = -y. \label{4.7.1}
\end{split}
\end{equation}
\end{Ex}
For $\alpha=0.5$, the exact solution is given as \cite{Cong},
\begin{equation}
\begin{split}
x(t)& = c_1E_\frac{1}{2}(\sqrt{t})-c_2^2\int_0^t (t-s)^{-1/2}E_{\frac{1}{2},\frac{1}{2}}(\sqrt{t-s})(E_\frac{1}{2}(-\sqrt{s}))^2\,\mathrm{d}s\\
y(t) & = c_2E_\frac{1}{2}(-\sqrt{t}). \label{4.7.2}
\end{split}
\end{equation}
Also the local stable manifold \cite{Cong} $S$ is given by,
\begin{equation}
x=-y^2\int_0^\infty e^{-s}(E_\frac{1}{2}(-\sqrt{s}))^2\,\mathrm{d}s=-\left(-1+\frac{4}{\pi}\right)y^2 \approx -0.27324 \,\,y^2\label{4.7.3}
\end{equation}

\begin{equation}
\therefore S=\{(x,y) : x=-0.27324 \,\,y^2\}. \label{4.7.4}
\end{equation}

If $(x(t),y(t))$ is solution of (\ref{4.7.1}) with initial condition $(x(0),y(0))=(c_1,c_2)$  on $S$ then $c_1=-0.27324 c_2^2$. Further, using asymptotic expansion of Mittag-Leffler function \cite{Podlubny},
\begin{equation}
 x(t) \approx -2(0.27324)c_2^2e^t+\frac{0.27324}{\sqrt{\pi}}t^{-1/2}c_2^2-\frac{2c_2^2e^t}{\pi}\int_{0}^{t}\frac{e^{-s}}{s}\mathrm{d}s
\end{equation}
for sufficiently large value of $t$.
Note that, R.H.S. does not tends to $0$ as $t\rightarrow \infty$ for any $c_2\ne 0$.\\
$\therefore$ $x(t) \nrightarrow 0$ as $t\rightarrow \infty$.\\
$\Rightarrow$ $S$ cannot be a local stable manifold.\\
 In the following Figure \ref{Fig 11}, we can see that the solution (\ref{4.7.2}) starting on $S$, in any small neighborhood of origin e.g. at $(-0.27324\times 10^{-20},\,\, 10^{-10})$ does not tends towards origin as $t\rightarrow \infty$. 
\begin{figure}[h]
   \begin{center}
             \includegraphics[width=0.6\textwidth]{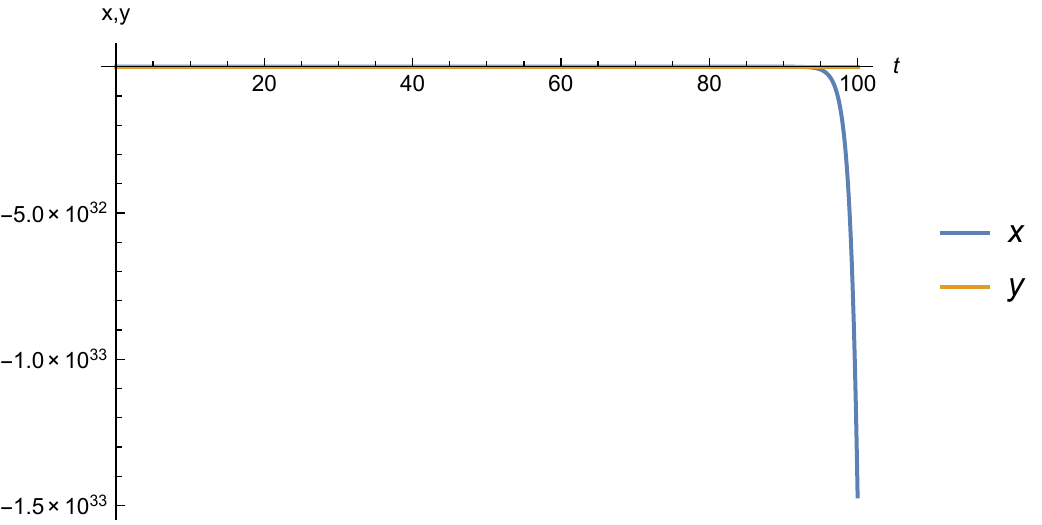}
            \caption{Solution curves (\ref{4.7.2}) with initial condition on $S$.}         
   \label{Fig 11}
   \end{center}  
          \end{figure}

\par Now, we prove that neither the parabola $x=my^2$ nor the parabola $y=mx^2$ are invariant under the flow of  system (\ref{4.7.1}).\\[0.15cm]
{\bf (I)}: \quad
Consider the parabola 
 $
x=my^2 
$. Differentiation of order $\alpha$ gives
\begin{equation*}
  {}_0^C\mathrm{D}_t^\alpha x=m \,\,{}_0^C\mathrm{D}_t^\alpha (y^2). \label{4.4.20}
\end{equation*}
By using generalized Leibniz rule for Caputo fractional derivative of order $0<\alpha<1$ and substituting $x=my^2$, we obtain
\begin{equation*}
\begin{split}
 (2m-1)y^2 -\frac{m y(0)t^{-\alpha}}{\Gamma(1-\alpha)}\big[y(t)-y(0)\big]-m \sum_{k=1}^{\infty}\binom{\alpha}{k}\big({}_0\mathrm{I}_t^{k-\alpha} y(t)\big)y^{(k)}(t)=0.
\end{split}
\end{equation*} 
This holds for all $t$ if and only if $2m-1=0$ and $m=0$, which is inconsistent.\\
$\Rightarrow$ There does not exist any $m\in\mathbb{R}$ such that $x=my^2$ is invariant under the flow of system (\ref{4.7.1}).\\

{\bf (II)}: \quad
Consider \begin{equation*}
y=mx^2. \label{4.4.21}
\end{equation*}
In this case, the tangency condition gives,
\begin{equation*}
\begin{split}
 m^2x^4-mx^2-mx -\frac{m x(0)t^{-\alpha}}{\Gamma(1-\alpha)}\big[x(t)-x(0)\big]-m \sum_{k=1}^{\infty}\binom{\alpha}{k}\big({}_0\mathrm{I}_t^{k-\alpha} x(t)\big)x^{(k)}(t)=0.
\end{split}
\end{equation*} 
Using the similar arguments, we can easily check that there does not exist any nonzero  $m\in\mathbb{R}$ such that $y=mx^2$ is invariant under the flow of system (\ref{4.7.1}).
  
\section{Conclusion}  \label{Sec6}
  We used tangency condition to propose the necessary and sufficient conditions for the existence of invariant straight lines and parabolas in the planar polynomial systems of ordinary differential equations. Further, we proved that the conditions for the invariance of linear subspaces in fractional order systems are same as their classical counterparts. Ample number of examples are provided to support the results.
  \par Important contribution of this work is the result showing the nonexistence of invariant manifolds (except linear subspaces) in fractional order systems. In particular, we have shown that any curve with curvature $>0$ cannot be invariant under the flow of fractional order system.

  \section*{Acknowledgment}
  S. Bhalekar acknowledges  the Science and Engineering Research Board (SERB), New Delhi, India for the Research Grant (Ref. MTR/2017/000068) under Mathematical Research Impact Centric Support (MATRICS) Scheme. M. Patil acknowledges Department of Science and Technology (DST), New Delhi, India for INSPIRE Fellowship (Code-IF170439).

\end{document}